\documentclass[reqno,oneside,12pt]{amsart}
\usepackage{fancyhdr}
\usepackage{xspace}
\usepackage{geometry}
\usepackage{tikz}
\usepackage{tikz-cd}
\usepackage{aligned-overset}
\tikzset{node distance=1.5cm, auto}
\usepackage{color}
\definecolor{darkgreen}{rgb}{0,0.45,0}
\usepackage[pagebackref,colorlinks,citecolor=darkgreen,linkcolor=darkgreen]{hyperref}
\usepackage[capitalize]{cleveref}
\usepackage{enumerate}
\usepackage[shortlabels]{enumitem}

\usepackage{graphicx}
\usepackage{enumerate}
\usepackage[all]{xy}
\usepackage{amsmath,amsfonts,amssymb}
\usepackage[latin9]{inputenc}

\def\sevdash{\scalebox{0.7}{\rotatebox[origin=c]{225}{$\bot$}}}
\def\swvdash{\scalebox{0.7}{\rotatebox[origin=c]{135}{$\bot$}}}

\def\ul{\underline}
\newcommand{\Hom}{{\sf Hom}}

\newcommand{\End}{{\sf End}}

\newcommand{\Mod}{{\sf Mod}}

\newcommand{\Mat}{{\sf Mat}}

\newcommand{\PMod}{{\sf PMod}}

\newcommand{\Vect}{{\sf Vect}}
\newcommand{\Alg}{{\sf Alg}}

\newcommand{\TCoalg}{{\sf TCoalg}}

\def\ot{\otimes}
\def\hot{\,\hat{\otimes}\,}

\def\CC{{\mathbb C}}

\def\NN{{\mathbb N}}

\newcommand{\Cc}{\mathcal{C}}

\newcommand{\Oo}{\mathcal{O}}

\newcommand{\selabel}[1]{\label{se:#1}}
\newcommand{\seref}[1]{Section~\ref{se:#1}}

\newcommand{\N}{{\mathbb N}}
\newcommand{\Z}{{\mathbb Z}}

\newcommand{\C}{{\mathbb C}}

\newcommand{\G}{\mathbb{G}}
\newcommand{\Grp}{\mathsf{Grp}}

\newcommand{\Set}{\mathsf{Set}}
\newcommand{\Mon}{\mathsf{Mon}}

\newcommand{\Nat}{\mathsf{Nat}}
\newcommand{\OG}{\mathcal{O}(G)}

\newtheorem{prop}{Proposition}[section]
\newtheorem{proposition}[prop]{Proposition}
\newtheorem{lemma}[prop]{Lemma} 
 
\newtheorem{corollary}[prop]{Corollary} 

\newtheorem{thm}[prop]{Theorem}
\newtheorem{theorem}[prop]{Theorem}
\newtheorem{conjecture}[prop]{Conjecture}

\theoremstyle{definition}

\newtheorem{definition}[prop]{Definition}

\newtheorem{example}[prop]{Example}
\newtheorem{examples}[prop]{Examples}

\newtheorem{remark}[prop]{Remark}
\newtheorem{remarks}[prop]{Remarks}

\newcommand{\benu}{\begin{enumerate}}
\newcommand{\enu}{\end{enumerate}}

\newcommand{\beqna}{\begin{eqnarray}}
\newcommand{\eqna}{\end{eqnarray}}
\newcommand{\beqnast}{\begin{eqnarray*}}
\newcommand{\eqnast}{\end{eqnarray*}}
\newcommand{\beqn}{\begin{equation}}
\newcommand{\eqn}{\end{equation}}
\newcommand{\beqnst}{\begin{equation*}}
\newcommand{\eqnst}{\end{equation*}}

\newcommand{\hotimes}{\,\hat{\otimes}\,}
\renewcommand{\hot}{\,\hat{\otimes}\,}

\usepackage{latexsym}

\usepackage{xspace}
\usepackage{amscd}
\usepackage{amssymb}
\usepackage{amsfonts}

\makeatother

\newcommand{\bema}{\left ( \begin{array}}
\newcommand{\ema}{\end{array} \right )}

\newcommand{\TVS}{\mathsf{TVS}}
\newcommand{\CTVS}{\mathsf{CTVS}}

\begin{document}

\title{A comonadicity theorem for partial comodules}

\author[E. Batista]{Eliezer Batista}
\address{Departamento de Matem\'atica, Universidade Federal de Santa Catarina, Brazil}
\email{ebatista@mtm.ufsc.br}
\author[W. Hautekiet]{William Hautekiet}
\address{D\'epartement de Math\'ematiques, Universit\'e Libre de Bruxelles, Belgium}
\email{william.hautekiet@ulb.be}
\author[J. Vercruysse]{Joost Vercruysse}
\address{D\'epartement de Math\'ematiques, Universit\'e Libre de Bruxelles, Belgium}
\email{joost.vercruysse@ulb.be}

\flushbottom

\begin{abstract} 
We show that the category of partial comodules over a Hopf algebra \(H\) is comonadic over \(\Vect_k\) and provide an explicit construction of this comonad using topological vector spaces. The case when \(H\) is finite dimensional is treated in detail. 
A study of partial representations of linear algebraic groups is initiated; we show that a connected linear algebraic group does not admit partiality. 
\end{abstract}

\maketitle

\tableofcontents

\section*{Introduction}

Over the last two decades, partial representations of groups \cite{doku} and Hopf algebras \cite{ABV} have proved to provide a valuable enrichment to usual representation theory. Several examples are known of groups that have equivalent categories of representations, whereas their partial representation categories are no longer equivalent. This shows that the category of partial representations is indeed a finer invariant than the usual representation category and hence could be a useful tool, for example, in the study of isocategorical groups \cite{isocat1, isocat2}. 

Intuitively, a partial representation of a group $G$ on a vector space $V$, associates to each element $g\in G$ not an automorphism of $V$ (as would be the case for a usual representation) but only an endomorphism, which however restricts to an isomorphism between the image of the action under $g^{-1}$ and the image of the action under $g$ itself. As such, one can say that $G$ acts on $V$ by `partial automorphisms'. Natural examples are obtained by restricting usual representations to (suitable) subspaces. Moreover, by the so-called globalization or dilation, one can show that any partial representation of a group \cite{Aba, DHSV} or a Hopf algebra \cite{ABV2} is obtained by restriction of a global one. Similar globalization results had been obtained earlier for, among others, partial actions of groups \cite{DokExe} and Hopf algebras \cite{AB} on algebras. This has lead to a lot of interest in the theory of partial actions and representations over the last decade, see the surveys \cite{Dok:survey, Bat} and the references therein.

Another crucial point in the theory of partial representations, is that these can be identified with usual modules over a suitably constructed algebra. In case of a (finite) group $G$, this algebra, which we denote as $k_{par}G$, is a groupoid algebra \cite{doku} and hence a weak Hopf algebra. In case of a Hopf algebra $H$, this algebra, which we denote as $H_{par}$, has the structure of a Hopf algebroid, not necessarily associated with a weak Hopf algebra \cite{ABV}.

As is well-known, representations of an algebraic group $G$ are described not as modules over the group algebra $kG$, but rather in terms of comodules over the coordinate algebra of $G$. The fundamental theorem of comodules states that any comodule equals the sum of its finite dimensional subcomodules. This result is at the heart of Tannaka duality \cite{Deligne}, which allows the reconstruction of an algebraic group from its monoidal category of finite dimensional representations. 

In view of this, in order to make partial representation theory applicable to algebraic groups, it is indispensable to consider partial comodules or corepresentations over a Hopf algebra, notions which were recently introduced in \cite{coreps}. A globalization theorem for partial comodules has been obtained recently in \cite{SV3}, within the more general framework of so-called `geometric partial comodules' \cite{HV}.
In spite of this, in \cite{coreps} the remarkable observation was made that partial comodules do not share the intrinsic finiteness observed in the usual theory of comodules over a coalgebra because, in general, they do not satisfy the fundamental theorem of comodules. An explicit example of a Hopf algebra that admits partial comodules that are different from the sum of their finite dimensional subcomodules is Sweedler's 4-dimensional Hopf algebra (see \cref{ex:irregular}). A consequence of this observation is that, for a given Hopf algebra $H$, one cannot guarantee the existence of a coalgebra $C$ such that the category of partial $H$-comodules is equivalent to the category of $C$-comodules (see \cref{prop:H4}). Those Hopf algebras admitting the existence of such a coalgebra are called regular. 
The existence of non-regular Hopf algebras, such as Sweedler's 4-dimensional Hopf algebra, shows that the theory of partial corepresentations is more involved than the theory of usual comodules and not a formal dualization of partial representation theory.

The first aim of this work is to show that, even if in general the category $\PMod^H$ of partial comodules over a Hopf algebra $H$ is not equivalent to a category of comodules over a coalgebra, it {\em is} equivalent to the Eilenberg-Moore category of a suitably constructed {\em comonad} $\C$ on the category of vector spaces. In other words, we show in \cref{se:comodulescomonadic} (more precisely in \cref{th:comonad}) that the forgetful functor from the category of partial comodules to the category of vector spaces is comonadic. 
Our proof is based on pure categorical reasoning (combining the  Special Adjoint Functor Theorem with Beck's precise (co)monadicity theorem) after analysing some important properties of the category of partial comodules over a Hopf algebra, where we show that this is a Grothendieck category whose colimits and finite limits can be computed in the underlying category of vector spaces (see \cref{prop:parcomcomplete} and \cref{th:grothendieck}).

It follows that a Hopf algebra $H$ is regular (in the sense explained above), if and only if the comonad $\C$ is of the form $-\ot C$ for some coalgebra $C$. Based on this, we study regularity in \cref{se:regular}, where we show in particular that a finite dimensional Hopf algebra $H$ is regular if and only if the Hopf algebroid $(H^*)_{par}$ characterizing the partial representations, is also finite dimensional (see \cref{le:Hfinreg}). This allows us to provide several classes of regular Hopf algebras, as well as a natural example of an infinite dimensional irregular Hopf algebra obtained from the universal quantum enveloping algebra \(U_q(\mathfrak{sl}_2)\) (see \cref{ex:quantum}).

\cref{se:algebriacgroups} is devoted to the theory of partial comodules over the coordinate algebra $\OG$ of a linear algebraic group $G$. Viewing $G$ as a representative functor $\mathbb{G}:\mathsf{Alg}_k \rightarrow \mathsf{Grp}$, we observe (see \cref{le:parrepalg}) that partial representations of $G$ on a vector space $V$, seen as natural transformations from $\mathbb{G}$ to the functor $\mathsf{End}_V\cong \Hom_k(V,V\ot -):\Alg_k\to \Grp$, are in one-to-one correspondence with partial comodules over the Hopf algebra $\OG$. Moreover, if $G$ is a connected group then all the partial comodules over $\OG$ are global, see \cref{th:connectedgroup}. This is a consequence of the non-degenerate dual pairing between the Hopf algebra $\OG$ and the universal enveloping algebra $U(\mathfrak{g})$ of the Lie algebra $\mathfrak{g}$ associated to the group $G$. By this duality, partial comodules over $\OG$ correspond to partial modules over $U(\mathfrak{g})$. As we already know that there are no partial  $U(\mathfrak{g})$-modules other than the global ones \cite{ABV}, the result follows. In \cref{subs:trivial} we construct some induced partial comodules of $\OG$ from global comodules of $\mathcal{O}(G^\circ)$, where $G^\circ$ is the connected component of the identity in $G$, using partial representations of the (finite) quotient group $G/G^\circ$. More specifically, we describe the structure of partial $\mathcal{O}(\mathsf{O}(2))$-comodules (\cref{se:orthogonal}) and present an example based on a group of monomial $3\times 3$-matrices (\cref{se:monomial}).

Finally, in \cref{se:explicit} we construct explicitly the right adjoint $R$ to the forgetful functor $U:\PMod^H \to \Vect_k$.
As a right adjoint of a forgetful functor, $R$ is intuitively understood as a {\em cofree construction}. Our strategy is to imitate, in a comodule-setting, the construction of cofree coalgebras and universal coalgebras as subcoalgebras of these \cite{BL,Hazel}. This is achieved by using some techniques from topological vector spaces and topological coalgebras \cite{takeuchi}, which we review in \cref{se:tvs}.
Explicitly, we associate to each vector space $V$ the largest subspace $R(V)$ of Takeuchi's topological cofree comodule $V\hotimes \hat{C}(H) =\prod_{n\geq 0} V\otimes H^{\otimes n}$, which inherits the structure of a partial $H$-comodule (see \cref{se:constructionR}).
Our explicit construction of the functor $R$ then implies also a description of the comonad $\mathbb{C} =UR:\Vect_k \rightarrow \Vect_k$, whose Eilenberg-Moore category coincides with the category of partial $H$-comodules by the results obtained in \cref{se:comodulescomonadic}. In the case were $H$ is a finite dimensional Hopf algebra, we find in \cref{pr:finitedimensional} that the functor underlying the comonad $\CC$ is representable, as it is naturally isomorphic to $\Hom_k ((H^*)_{par} , -)$, where $(H^*)_{par}$ is the ``partial'' Hopf algebroid associated to the dual Hopf algebra $H^*$. In general however, the comonad $\CC$ cannot be induced by a coalgebra, nor by an algebra, which shows that the comonadic approach presented in this paper is unavoidable.

We end this paper with an outlook to some future directions for research. 

\subsection*{Notational conventions}
In this article, unless stated explicitly otherwise, \(k\) is a field and \(H\) a Hopf algebra over \(k\), whose multiplication is denoted by \(\mu : H \otimes H \to H\), unit by \(u : k \to H\), comultiplication by \(\Delta : H \to H \otimes H\), counit by \(\epsilon : H \to k\) and antipode by \(S : H \to H\). We will make use of the Sweedler notation for the comultiplication
\[\Delta(h) = h_{(1)} \otimes h_{(2)}\]
for \(h \in H\).

When \(A\) is an object in a category \(\mathcal{C}\), the identity morphism of \(A\) is also denoted by \(A\). Composition of morphisms \(f : A \to B\) and \(g : B \to C\) is denoted \(g\circ f=gf : A \to C\). 

For a linear map \(\rho : M \to M \otimes H,\) we write \(\rho^0 = M\). Denote by induction $$\rho^n:= (\rho \ot H^{\ot n-1}) \rho^{n-1}.$$ Also we adopt the Sweedler notation to write $\rho(m)=m^{(0)}\ot m^{(1)}$.

\section{The category of partial comodules}
\label{se:comodulescomonadic}

In this section we recall the definitions and some important results about partial modules and comodules over a Hopf algebra. Next, we will prove that the category of partial comodules
\(\PMod^H\) has very similar properties as the category of (global) comodules, in particular both are Grothendieck categories. Nevertheless, we will also show that there exist Hopf algebras for which the category of partial comodules cannot be equivalent to the category of comodules for over any coalgebra. However, we show that the category of partial comodules is comonadic over $\Vect_k$.

\subsection{Partial modules and comodules}
\label{se:generalities}

Let us recall from \cite{ABV}, the definition of a partial module over a Hopf algebra $H$, as well as the result which states that the category of partial modules over \(H\) is equivalent to the category of (usual, global) modules over the so-called partial ``Hopf'' algebra \(H_{par}\).

\begin{definition}
    A \textit{left partial module} over \(H\) is a couple \((M, \pi)\) where \(M\) is a \(k\)-vector space and \[\pi: H \otimes M \to M : h \otimes m \mapsto h \cdot m\] a linear map such that
    \begin{enumerate}[(PM1)]
        \item \(1_H \cdot m = m\) for all \(m \in M\);
        \item \(h \cdot (k_{(1)} \cdot (S(k_{(2)}) \cdot m)) = hk_{(1)} \cdot (S(k_{(2)}) \cdot m)\) for all \(h, k \in H, m \in M\);
        \item \(h_{(1)} \cdot (S(h_{(2)}) \cdot (k \cdot m)) = h_{(1)} \cdot (S(h_{(2)})k \cdot m)\)  for all \(h, k \in H, m \in M\). 
    \end{enumerate}
    A morphism between partial modules \((M, \pi)\) and \((M', \pi')\) is a linear map \(f : M \to M'\) such that \(f  \pi = \pi' (H \otimes f)\). The category of left partial modules over \(H\) is denoted by \({_H} \PMod\).
\end{definition}

\begin{theorem}[{\cite[Corolloray 5.3]{ABV}}]\label{Hpar}
The category \({_H} \PMod \) of left partial \(H\)-modules is isomorphic to the category of left \(H_{par}\)-modules \({_{H_{par}}} \Mod\), where \(H_{par}\) is defined as the quotient of the tensor algebra over \(H\) 
    \[T(H) = \langle [h] \mid h \in H \rangle\]
    by the ideal generated by the elements
    \begin{gather*}
        1_{T(H)} - [1_H], \\
        [h][k_{(1)}][S(k_{(2)})] - [hk_{(1)}][S(k_{(2)})],\\
        [h_{(1)}][S(h_{(2)})][k] - [h_{(1)}][S(h_{(2)})k]
    \end{gather*}
for \(h, k \in H\).
\end{theorem}


Partial comodules, the dual notion to partial modules, were introduced in \cite{coreps}.
\begin{definition}
	An (algebraic) right partial \(H\)-comodule is a \(k\)-vector space \(M\) endowed with a \(k\)-linear map \(\rho : M \to M \otimes H,\) satisfying
	\begin{enumerate}[(PCM1)]
		\item \label{PCM1} $(M\otimes \epsilon) \rho =M$;
		\item  \label{PCM2} $(M\otimes H \otimes \mu)(M\otimes H \otimes H \otimes S)(M\otimes \Delta \otimes H) (\rho \otimes H )\rho =(M\otimes H \otimes \mu)(M\otimes H \otimes H \otimes S)$\\
		$(\rho \otimes H \otimes H)(\rho \otimes H)\rho$;
		\item \label{PCM3} $(M\otimes \mu \otimes H)(M\otimes H \otimes S \otimes H) (M\otimes H \otimes \Delta)(\rho \otimes H)\rho =(M\otimes \mu \otimes H)(M\otimes H \otimes S \otimes H)$\\
		$(\rho \otimes H \otimes H)(\rho \otimes H)\rho$.
	\end{enumerate}
	The first axiom tells that \(\rho\) is counital, while \ref{PCM2} and \ref{PCM3} express the \textit{partial coassociativity}. 
	As shown in \cite[Lemma 3.3]{coreps}, axioms \ref{PCM2} and \ref{PCM3} are equivalent to
	\begin{enumerate}[(PCM1)]
	\setcounter{enumi}{3}
	    \item  \label{PCM4} $(M\otimes H\otimes \mu)(M\otimes H \otimes S \otimes H)(M\otimes \Delta \otimes  H)(\rho \otimes H)\rho=(M\otimes H\otimes \mu)(M\otimes H \otimes S \otimes H)$\\
	    $(\rho \otimes H \otimes H)(\rho \otimes H)\rho$;
	    \item \label{PCM5} $(M\otimes \mu \otimes H)(M\otimes S\otimes H \otimes H)(M\otimes H \otimes \Delta)(\rho \otimes H)\rho =(M\otimes \mu \otimes H)(M\otimes S\otimes H \otimes H)$\\
	    $(\rho \otimes H \otimes H)(\rho \otimes H)\rho$.
	\end{enumerate}

    A morphism between partial comodules \((M, \rho)\) and \((N, \sigma)\) is a linear map \(f : M \to N\) such that \((f \otimes H) \rho = \sigma f\). 
The category of right partial comodules over \(H\) is denoted as \(\mathsf{PMod}^H\). 
\end{definition}

Partial comodules as defined above, should not be confused with the more general notion of `geometric partial comodule' as introduced in \cite{HV}. In \cite{SV3}, it was shown that partial comodules in the sense above are a particular kind of geometric partial comodules.

Partial comodules and partial modules are related by the following result.

\begin{thm}[{\cite[Theorem 4.14]{coreps}}]
\label{thm:duality}
Let $H$ be a Hopf algebra over a field such that $H^{\circ}$ is dense in $H^*$, $M$ a $k$-vector space and $\rho:M\rightarrow M\otimes H$ a linear map. then the following statements are equivalent:
\begin{enumerate}
    \item[(i)] $(M,\rho)$ is a right partial $H$-comodule.
    \item[(ii)] $(M,\cdot )$ is a left partial $H^\circ$-module, where $\cdot \! \! : H^\circ \otimes M\rightarrow M$ is defined by $h^\ast \cdot m=(M\otimes h^\ast ) \rho (m)$.
\end{enumerate}
\end{thm}

\begin{example}
Consider the group algebra \(kC_3 = \langle 1, g, g^2\rangle\). If $k$ is a field whose characteristic is different from $3$, then it is well-known that $kC_3$ is semi-simple by Maschke's theorem. If moreover $k$ contains a primitive third root of unity $\xi$, then $C_3$ has up to isomorphism exactly $3$ irreducible representations, that are all $1$-dimensional and determined by the morphisms $\chi_i:G\to k$, $i=1,2,3$ given by
$$\chi_1(g)=1, \quad \chi_2(g)=\xi, \quad \chi_3(g)=\xi^2.$$
As a consequence, $kC_3\cong k^3$ as $k$-algebras.

The partial Hopf algebra \(kC_{3, par}\) on the other hand is 8-dimensional and has a basis
\[\{[1], [g], [g]^2, [g^2], [g^2]^2, [g][g^2], [g^2][g], [g]^3\}.\]
The subalgebra \(A\) is generated by \([1], [g][g^2]\) and \([g^2][g]\) and is three-dimensional. As in the global case, if \(\mathrm{char}(k) \neq 3,\) then \(kC_{3, par}\) is again a semisimple algebra, and if $k$ contains a primitive third root of unity, $kC_{3,par}$ is moreover isomorphic to \(k^4 \times {\sf Mat}_{2\times 2}(k)\) \cite{doku}. In other words, there are 4 isomorphism classes of one-dimensional irreducible partial modules over $kC_3$, and 1 two-dimensional \cite{DZ}. The one-dimensional irreducible partial representations are the three global ones given above, and the following partial representation $\chi_4:C_3\to k$ given by
$$\chi_4(1)=1, \chi_4(g)=\chi_4(g^2)=0.$$
The two-dimensional irreducible partial representation $\chi_5:C_3\to \Mat_{2\times 2}(k)$ is given by
\[
\chi_5(1)=\left(\begin{array}{cc}1 & 0 \\ 0 & 1 \end{array}\right) \quad
\chi_5(g)=\left(\begin{array}{cc}0 & 0 \\ 1 & 0 \end{array}\right) \quad
\chi_5(g^2)=\left(\begin{array}{cc}0 & 1 \\ 0 & 0 \end{array}\right) 
\]

As for any finite dimensional Hopf algebra $H$, the category of partial comodules over $H$ is isomorphic to the category of partial modules over $H^*$, and vice versa. If the field $k$ has a primitive third root of unity, then, from the partial modules over $kC_3$, one can construct the irreducible partial comodules over $(kC_3)^*$. But \(kC_3 \cong (kC_3)^*\) as Hopf algebras by means of the isomorphism
\[
\phi (1) =p_1 +p_g +p_{g^2}, \qquad \phi (g) =p_1 +\xi p_g +\xi^2 p_{g^2} , \qquad \phi (g^2) =p_1 +\xi^2 p_g +\xi p_{g^2} ,
\]
where we denoted the basis of \((kC_3)^*\) dual to \(\{1, g, g^2\}\) by $\{ p_1 , p_g , p_{g^2} \}$.


This way, we find the irreducible partial comodules over $kC_3$, that is, partially $C_3$-graded vector spaces:
\begin{itemize}
    \item Three \(1\)-dimensional global $kC_3$ comodules: $(ke_1, \rho_1 )$, $(ke_1, {\rho}_2 )$ and $(ke_1,{\rho}_3 )$, given by
    \begin{eqnarray*}
    {\rho}_1 (e_1) & = & e_1 \otimes 1 ,\\
    {\rho}_2 (e_1) & = & e_1 \otimes g ,\\
    {\rho}_3 (e_1) & = & e_1 \otimes g^2 .
    \end{eqnarray*}
    \item One \(1\)-dimensional partial $kC_3$ comodule: $(ke_1 , {\rho}_4)$, given by
    \[
    {\rho}_4 (e_1)= \frac{1}{3} e_1 \otimes (1+g+g^2) .
    \]
    \item One \(2\)-dimensional partial $kC_3$ comodule $(k^2, {\rho}_5)$, given by
    \begin{eqnarray*}
    {\rho}_5 (e_1) & = & \frac{1}{3} e_1 \otimes (1+g+g^2) + \frac{1}{3} e_2 \otimes (1+\xi^2 g +\xi g^2) ,\\
    {\rho}_5 (e_2) & = & \frac{1}{3} e_1 \otimes (1+\xi g +\xi^2 g^2) +\frac{1}{3} e_2 \otimes (1+g+g^2) ,
    \end{eqnarray*}
    where $\{ e_1 , e_2 \} $ is the basis of the vector space $k^2$.
\end{itemize}


\end{example}



%

\subsection{\texorpdfstring{\(\PMod^H\)}{PMod} is a Grothendieck category and comonadic over \texorpdfstring{$\Vect_k$}{Vect}}
\label{Grothendieck}


\begin{proposition}\label{prop:parcomcomplete}
Let $H$ be a Hopf algebra over a field $k$.
The category $\PMod^H$ of partial $H$-comodules is complete and cocomplete and the forgetful functor $U:\PMod^H\to \Vect$ preserves and reflects (i.e. creates) all colimits and finite limits. Moreover, $\PMod^H$ is an abelian category. In particular, monomorphisms, epimorphisms and isomorphisms in $\PMod^H$ are respectively injecitve, surjective and bijective morphisms.
\end{proposition}

\begin{proof}
	The proofs of these facts are standard. For the interested reader, let us give an explicit proof for the existence of finite limits. Let \(F : J \to \PMod^H\) be a diagram of shape \(J\), where \(J\) is some finite index category. Let \(UF\) the corresponding diagram in \(\Vect_k\) and let \((M, \varphi_X)\) be its limit. Then \((M \otimes H, \varphi_X \otimes H)\) is the limit of the diagram \(UF(-) \otimes H\) because the functor \(- \otimes H\) preserves finite limits, and the collection of linear maps \(\rho_X : X \to X \otimes H\) for all objects \(X\) in \(F\) defines a morphism of diagrams \(UF \to UF(-) \otimes H\). Indeed, every morphism in \(F\) is a morphism of partial comodules, hence it commutes with the partial coactions. Considering the linear maps \(\rho_X \varphi_X : M \to X \otimes H\) for every \(X\) in \(F\), the universal property of \(M \otimes H\) gives a unique linear map \(\rho_M : M \to M \otimes H\) such that for every object \(X\) in \(F\), \(\rho_X \varphi_X = (\varphi_X \otimes H)\rho_M\).
	
	The maps \(X \otimes \epsilon : X \otimes H \to X\) induce a morphism \(UF(-) \otimes H \to UF\) and hence we have a unique morphism \(g : M \otimes H \to M\) such that \(\varphi_X  g = (X \otimes \epsilon) (\varphi_X \otimes H)\). Clearly \(g = M \otimes \epsilon\) and since for all \(X,\) \((X \otimes \epsilon) \rho_X = X\), we find that \((M \otimes \epsilon)  \rho_M= M\) (indeed, the identity on \(M\) is the unique linear map \(f: M \to M\) such that \(\varphi_X f = \varphi_X\) for all \(X\)). Hence \(\rho_M\) is counital. 
	\[\begin{tikzcd}
		M \arrow{r}{\rho_M} \arrow[bend left = 30]{rrr}{M} \arrow{d}{\varphi_X} & M \otimes H \arrow[dashed]{rr}{g = M \otimes \epsilon}\arrow{d}{\varphi_X \otimes H} && M \arrow{d}{\varphi_X}\\
		X \arrow{r}{\rho_X} \arrow[bend right = 30, swap]{rrr}{X} & X \otimes H \arrow{rr}{X \otimes \epsilon} && X
	\end{tikzcd}\]
	The partial coassociativity is proved in a similar way (consider also the diagrams \(F(-) \otimes H \otimes H\) and \(F(-) \otimes H \otimes H \otimes H\)). We can conclude that \((M, \rho_M)\) is a partial comodule.

	This partial comodule is indeed the limit of the diagram \(F\) in \(\PMod^H,\) for if \((N, \psi_X)\) is a cone on \(F,\) there exists a unique linear map \(h : N \to M\) such that \(\varphi_X h = \psi_X\) for all \(X\). Then \((U(N) \otimes H, \psi_X \otimes H)\) is a cone on \(UF(-) \otimes H\), and \(h \otimes H\) is the unique linear map \(N \otimes H \to M \otimes H\) from the universal property of \(M \otimes H\). Since both \((N, (\varphi_X \otimes H) (h \otimes H) \rho_N)\) and \((N, (\varphi_X \otimes H) \rho_M h)\) are cones on \(UF(-) \otimes H,\) it follows by universality of \(M \otimes H\) that \((h \otimes H)  \rho_N = h \rho_M\). So \(h\) is a morphism of partial comodules and we have shown that \(U\) preserves and reflects finite limits.
	\[
	\begin{tikzcd}
		N \arrow{r}{\rho_N} \arrow[bend right = 45, swap]{dd}{\psi_X} \arrow[dashed]{d}{h} & N \otimes H \arrow[bend left = 45, shift left = 2 ex]{dd}{\psi_X \otimes H} \arrow[dashed, swap]{d}{h \otimes H}\\
		M \arrow{r}{\rho_N} \arrow{d}{\varphi_X} & M \otimes H \arrow[swap]{d}{\varphi_X \otimes H}\\
		X \arrow{r}{\rho_N} & X \otimes H
	\end{tikzcd}
	\]
	
	The proof that \(U\) preserves and reflects  colimits is carried out in a similar way (remark that the functor \(- \otimes H : \Vect_k \to \Vect_k\) preserves all colimits since it is a left adjoint). Now it follows immediately that \(\PMod^H\) is abelian, since binary biproducts, kernels and cokernels are computed as in \(\Vect_k\).
	
    Arbitrary limits in \(\PMod^H\) will be discussed in \cref{rm:products}.
\end{proof}

\begin{lemma}
	\label{le:countabledim}
	Let \(H\) be a Hopf algebra, \((M, \rho)\) a partial comodule over \(H\) and \(m \in M\). Then \(m\) is contained in a partial subcomodule of \(M\) of countable dimension. 
\end{lemma}
\begin{proof}
	Let \(\{h_i \mid i \in I\}\) a basis for \(H\). Write for \(n \in \N\)
	\[\rho^n(m) = \sum_{\alpha \in I^n} m_\alpha  \otimes h_\alpha\]
	where \(m_\alpha \in M\) is zero for all but finitely many \(\alpha = (\alpha_1, \dots, \alpha_n)\) and \(h_\alpha = h_{\alpha_1} \otimes \cdots \otimes h_{\alpha_n}
 \in H^{\otimes n}\). Let \(N\) be the vector space generated by \(\{m_\alpha \mid \alpha \in I^n, n \in \N\}\) (which is a countable set). It is easy to see that \(N\) is a partial subcomodule. 
\end{proof}

\begin{lemma}
	\label{le:generator}
	The category of partial comodules over \(H\) has a generator.
\end{lemma}
\begin{proof}
	It suffices to show that \(\PMod^H\) has a set of generators, since then the direct sum of all objects in this set provides a generator. 
	By \cref{le:countabledim}, every partial comodule is the sum of its countable dimensional partial subcomodules. 
	For every countable dimension, there exists up to isomorphism a unique vector space. The isomorphism classes of partial comodule structures on this vector space form a set since they are a subset of \(\Hom_k(M, M \otimes H)\). Hence the collection of isomorphism classes of partial comodules of countable dimension is a set. 
\end{proof}


\begin{theorem}\label{th:grothendieck}
	\(\PMod^H\) is a Grothendieck category. 
\end{theorem}
\begin{proof}
	By \cref{prop:parcomcomplete}, colimits in \(\PMod^H\) and are computed as in \(\Vect_k\). Since \(\Vect_k\) is a Grothendieck category, the colimit functor is exact, and hence the same holds in \(\PMod^H\).	
	The result follows now by combining \cref{prop:parcomcomplete} and \cref{le:generator}.
\end{proof}

\begin{remark}
\label{rm:products}
    Since any Grothendieck category is complete, \(\PMod^H\) is too. Arbitrary products are constructed as follows: consider a family of partial comodules $(M_i,\rho_i)_{i\in J}$. We can then consider the product $\prod_{i\in I} M_i$ in the category of vector spaces, and there is an obvious inclusion
$$\left(\prod_{i\in I} M_i\right)\ot H\subset \prod_{i\in I} (M_i\ot H)$$
We define $P$ as the sum of all subspaces $N\subset\prod_{i\in I} M_i$ for which the image under the map $\prod_{i\in I} \rho_i: \prod_{i\in I} M_i\to \prod_{i\in I} (M_i\ot H)$ lies in $N\ot H\subset \left(\prod_{i\in I} M_i\right)\ot H$. Then $P$ is a partial comodule with coaction induced by $\prod_{i\in I} \rho_i$ and it satisfies the universal property of the product in $\PMod^H$.

Let us remark that from this construction, it is clear that in general the product $P$ in $\PMod^H$ is strictly smaller than the product in $\Vect$ and hence the forgetful functor $U$ does not preserve all limits.
\end{remark}

\begin{corollary}
	\label{th:comonad}
	Let $H$ be a Hopf algebra over a field $k$.
	Then the forgetful functor $\PMod^H\to \Vect_k$ is comonadic,
	in other words, $\PMod^H$ is equivalent to the Eilenberg-Moore category ${\Vect_k}^\CC$ with respect to some comonad $\CC$ on $\Vect_k$. 
\end{corollary}
\begin{proof}
By \cref{prop:parcomcomplete}, we know that $\PMod^H$ is cocomplete and since it is a Grothendieck category by \cref{th:grothendieck}, it is also well-copowered (i.e.\ the class of quotient objects of any object forms a set) and has a generator. Moreover by \cref{prop:parcomcomplete} the forgetful functor preserves all colimits. Therefore, we can apply the Special Adjoint Functor Theorem (see e.g.\ \cite[Exercise 1.9.20]{BarWel}) and we find that the forgetful functor \(U : \PMod^H \to \Vect_k\) has a right adjoint. By \cref{prop:parcomcomplete} we also know that 
\(\PMod^H\) has equalizers, $U$ preserves equalizers and reflects isomorphisms.
By Beck's precise comonadicity theorem (see e.g.\ \cite[Section 3.3, Theorem 10]{BarWel} or \cite{GT} for a comonadic formulation) we can conclude that \(\PMod^H\) is comonadic over \(\Vect_k\).
\end{proof}

The previous main result of this paper, raises the question if we can describe explicitly the comonad $\CC$, and moreover if this comonad could be induced by a $k$-coalgebra. In the next subsection, we will show that the answer to the second question is negative in general. An explicit construction of the comonad $\CC$ will be given in \seref{explicit}.

\subsection{Partial comodules are not equivalent to comodules}\label{se:parcomnotcom}

When \(C\) is a \(k\)-coalgebra and \(M\) a \(C\)-comodule, then the fundamental theorem of comodules asserts that every element of \(m\) is contained in a finite dimensional subcomodule. It was shown in \cite{coreps} that a similar theorem does not hold for partial comodules. Let us start by making some further observations about the example given in \cite{coreps}. 

\begin{example}
    \label{ex:irregular}
    Consider Sweedler's Hopf algebra \(H_4\), which is generated (as a vector space) by the unit \(1\), the grouplike element \(g\) and the skew-primitive elements \(x\) and \(y\) satisfying
    \begin{align*}
        g^2 = 1; &\qquad y = xg = -gx \\
        \Delta(x) = g \otimes x + x \otimes 1; &\qquad \Delta(y) = 1 \otimes y + y \otimes g;
        \\ S(x) = -y; &\qquad S(y) = x.
    \end{align*}
    Write \(e = \frac{1 + g}{2}\). Let \(k[z]\) the polynomial ring in the variable \(z\) and consider the linear map 
\[\rho : k[z] \to k[z] \otimes H_4 : z^n \mapsto z^n \otimes e + z^{n + 1} \otimes y.\]
Then \(X = (k[z], \rho)\) is a partial \(H_4\)-comodule.

Let us observe that any partial subcomodule of \(X\) is isomorphic to \(X\) itself. Let \(M \leq X\) be a partial subcomodule and \(f \in M\). Then \(z f \in M\) because \(zf \otimes y = \rho(f) - f \otimes e \in M \otimes H_4\). Therefore, $M$ is an ideal in $k[z]$, and since $k[z]$ is a PID, $M$ is a principal ideal generated by its monic element of lowest degree $f$. We can conclude that
\[\varphi : X \to M : z^n \mapsto z^n f\]
is an isomorphism of partial comodules.

In particular, denoting by \(W_n\) the partial subcomodule of \(X\) generated by \(z^n\), we obtain an infinite descending chain of partial subcomodules
\[X = W_0 \supsetneq W_1 \supsetneq \cdots \supsetneq W_n \supsetneq \cdots\]
From this observation, we will now show that \(\PMod^{H_4}\) is not equivalent to \(\Mod^C\) for any coalgebra \(C\).
\end{example}


Recall (see e.g.\ \cite[Section V.3]{Stenstrom}) that an object $X$ in a Grothendieck category $\Cc$ is called finitely generated if and only if one of the following two equivalent conditions holds: 
\begin{enumerate}[(1)]
\item if $X$ can be written as a filtered colimit of subobjects, then $X$ equals one of these subobjects;
\item the functor $\Hom_{\mathcal{C}}(X, -) : \mathcal{C} \to \mathsf{Set}$ preserves filtered colimits of monomorphisms. 
\end{enumerate}
By the fundamental theorem of comodules (which states that any comodule is the filtered colimit (i.e.\ the sum) of its finite dimensional subcomodules) over a coalgebra $C$, it is clear that a comodule is finitely generated as an object in $\Mod^C$ if and only if it is finite dimensional. The following Lemma in combination with \cref{ex:irregular} shows that there can exist finitely generated objects in the category of partial comodules over a Hopf algebra that are not finite dimensional.

\begin{lemma}
    \label{le:fingenparco}
    Let \(M\) be a partial \(H\)-comodule and suppose there exists a finite set \(F \subseteq M\) for which the only partial subcomodule of \(M\) containing \(F\) is \(M\) itself. Then \(M\) is a finitely generated object in the Grothendieck category $\PMod^H$.
\end{lemma}
\begin{proof}
Let \(N = \varinjlim N_i\) be a filtered colimit of monomorphisms in \(\PMod^H\) and \(f : M \to N\) a morphism of partial comodules. Since the colimit is filtered and $F$ is finite, there is an \(i\) such that \(f(F) \subseteq N_i\). If \(f(M) \nsubseteq N_i\), \(f^{-1}(N_i)\) would be a proper partial subcomodule of \(M\) containing \(F\), hence \(f(M) \subseteq N_i\), which shows that the canonical injection 
\[\varinjlim \Hom (M, N_i) \to \Hom(M, N)\]
is surjective. 
\end{proof}

%
%

\begin{proposition}
    \label{prop:H4}
    There does not exist a coalgebra \(C\) such that \(\PMod^{H_4}\) is equivalent to \(\Mod^C\). 
\end{proposition}
\begin{proof}
    As explained above, every finitely generated object in \(\Mod^C\) has a finite dimensional underlying vector space; in particular, every ascending or descending  chain of subobjects of a finitely generated object stabilizes (i. e. every finitely generated object is noetherian and artinian). However, the partial comodule \(X\) from Example \ref{ex:irregular} is finitely generated but not artinian. 
\end{proof}

\subsection{Regularity}\label{se:regular}
Example \ref{ex:irregular} shows that \cref{le:countabledim} cannot be strengthened in general. However, for some Hopf algebras a fundamental theorem for partial comodules does hold. This was the motivation for the introduction of the notion of regular partial comodule in \cite{coreps}. Let us recall the exact definition.

\begin{definition}
	A partial comodule \(M\) over \(H\) is called \textit{regular} if \(M\) is the sum of its finite dimensional partial subcomodules. 
	A Hopf algebra \(H\) is {\em regular} if every partial comodule over \(H\) is regular.
\end{definition}

We have the following result, which is basically a reformulation of the results from \cite[Section 5.1]{coreps}.


\begin{proposition}\label{regularH}
	A Hopf algebra \(H\) is regular if and only if the comonad $\CC$ from \cref{th:comonad} is of the form $\CC\cong -\ot H^{par}$ for some coalgebra $H^{par}$. 
\end{proposition}

\begin{proof}
The Eilenberg-Moore category of a comonad of the form $\CC\cong -\ot H^{par}$ is isomorphic to the category of $H^{par}$-comodules. Furthermore, by the comonadicity of the forgetful functor $U:\PMod^H\to \Vect_k$ (see \cref{th:comonad}), the equivalence of $\PMod^H$ and the Eilenberg-Moore category $\Vect_k^\CC$ commutes with the forgetful functors to $\Vect_k$. Combining this, we find that $\CC$ is of the form $-\ot H^{par}$ for some coalgebra $H^{par}$ if and only if
	there exists an equivalence \(\PMod^H \simeq \Mod^{H^{par}}\) that respects the forgetful functor to \(\Vect_k:\)
	\[\begin{tikzcd}
		\PMod^H \arrow[dash]{rr}{\simeq} \arrow{dr} && \Mod^{H^{par}} \arrow{dl} \\
		& \Vect_k &
		\end{tikzcd}\]
In \cite[Section 5.1]{coreps} an explicit construction of such a coalgebra $H^{par}$ for a regular Hopf algebra $H$ is given. If conversely an equivalence of categories as in the above diagram exists, then by the fundamental theorem of comodules for $H^{par}$, we find that also every partial $H$-comodule is the sum of its finite dimensional subcomodules.
\end{proof}

Before stating the next result, we provide the following lemma, which has its own interest.

\begin{lemma}\label{equivcatfindim}
If $C$ is a $k$-coalgebra and $R$ a $k$-algebra such that there is an equivalence of categories $\Mod_R\cong \Mod^C$, then $R$ and $C$ are finite dimensional.
\end{lemma}

\begin{proof}
This is a direct application of \cite[Corollary 3.7]{Ver:equiv}. Denote the equivalence by $F:\Mod_R\to \Mod^C$ and $F(R)=M$. Then by \cite[Corollary 3.7 (i)]{Ver:equiv}, $M$ is finite dimensional (in our case, the base ring $A$ from \cite[Corollary 3.7]{Ver:equiv} is the ground field $k$). Therefore $C\cong M^*\ot_T M$ (by \cite[Corollary 3.7 (iii)]{Ver:equiv}) and $R\cong \End^C(M)$ (by \cite[Corollary 3.7 (iv)]{Ver:equiv}) are finite dimensional.
\end{proof}

Let \(H\) be a finite dimensional Hopf algebra, \(H^*\) its dual Hopf algebra and consider \((H^*)_{par}\). Combining \cref{Hpar} and \cref{thm:duality}  we obtain an equivalence
\[\PMod^H \simeq {_{(H^*)_{par}}} \Mod.\]
Explicitly, when \((M, \rho)\) is a partial comodule, then 
\[(H^*)_{par} \otimes M \to M : [h^*] \otimes m \mapsto (M \otimes h^*) \rho(m)\]
defines an \((H^*)_{par}\)-module structure on \(M\). 
Conversely, let \(\{h_i, h_i^*\}\) a dual basis for \(H\), then when \(M\) is an \((H^*)_{par}\)-module, 
\[\rho : M \to M \otimes H : m \mapsto \sum_i [h_i^*] \cdot m \otimes h_i\]
defines a partial comodule structure on \(M\). 

\begin{theorem}
    	\label{le:Hfinreg}
	Let $H$ be a finite dimensional Hopf algebra. Then the following statements are equivalent:
\begin{enumerate}[(i)]		
\item \(H\) is regular; 
\item \((H^*)_{par}\) is finite dimensional;
\item there exists a coalgebra \(C\) such that \(\PMod^H \simeq \Mod^C\). 
\end{enumerate}
Under these equivalent conditions, we have a natural isomorphism $H^{par}\cong ((H^*)_{par})^*$.
\end{theorem}

\begin{proof}
$\ul{(ii)\Rightarrow(i)}$.
Suppose that \(H\) is not regular. Then there exists an infinite dimensional cyclic partial comodule \(M\), generated by \(m \in M\).
By the observations above, $M$ is also an infinite dimensional cyclic partial $(H^*)_{par}$-module, which is in contradiction with $(H^*)_{par}$ being finite dimensional.
	
$\ul{(i)\Rightarrow (iii)}$.
This follows directly from  \cref{regularH}, taking $C=H^{par}$.

$\ul{(iii)\Rightarrow(ii)}$
By the remarks preceding the Theorem, we find that $\Mod^{C}\simeq \PMod^H \simeq{_{(H^*)_{par}}\Mod}$. Then by \cref{equivcatfindim} we can immediately deduce that $(H^*)_{par}$ is finite dimensional.
%

	The last statement follows directly from \cite[Corollary 5.4]{coreps}.
\end{proof}

In \cite{doku}, it was shown that \(kG_{par}\) is a finite dimensional algebra for any finite group \(G\), hence we have the following corollary. 

\begin{corollary}
\label{cor:dualgroupregular}
	Let \(G\) be a finite group. Then \(kG^*\) is regular.
\end{corollary}

\begin{corollary}
	Let k be algebraically closed and of characteristic 0 and let \(G\) be a finite abelian group. Then \(kG\) is regular.
\end{corollary}
\begin{proof}
    Under the stated conditions, \(kG^*\) is isomorphic to \(kG\) as Hopf algebras, hence the result follows from \cref{cor:dualgroupregular}.
\end{proof}

It is unclear whether \(kG\) is regular for non-abelian finite groups. We confirmed with a computer that \((kS_3)^*_{par}\) is 51-dimensional and \((kD_8^*)_{par}\) and \((kQ_8^*)_{par}\) are both 180-dimensional, hence \(kS_3, kD_8, kQ_8\) are all regular. In fact, we expect all finite dimensional cosemisimple Hopf algebras to be regular.

\begin{conjecture}
	\label{conj:regular}
	If \(H\) is finite dimensional and cosemisimple, then \(H\) is regular. 
\end{conjecture}

By \cref{le:Hfinreg}, the conjecture above is equivalent to ``If $H$ is finite dimensional and cosemisimple, then $(H^*)_{par}$ is finite dimensional.'' On the other hand, recall that $H$ is finite dimensional and cosemisimple if and only if $H^*$ is semisimple. Hence the conjecture can furthermore be reformulated as ``If $H^*$ is semisimple, then $(H^*)_{par}$ is finite dimensional''. 
In \cite[Conjecture 2.8]{ABV2} the was conjectured that if $H$ is finite dimensional and semisimple, then $H_{par}$ would also be finite dimensional and semisimple. Therefore, \cref{conj:regular} can be seen as an intermediate step towards this stronger conjecture.
If the characteristic of \(k\) is zero, then every \(H\) is semisimple if and only if it is cosemisimple (see \cite{LR}), hence in this case we expect every semisimple Hopf algebra to be regular. 


\begin{proposition} If $H$ is regular then every Hopf subalgebra of $H$ is also regular. \end{proposition}
\begin{proof}
    Let \(H'\) be a Hopf subalgebra of \(H\) and suppose \(H'\) is not regular. Then there exists a nonregular partial \(H'\)-comodule \((V, \rho)\). Denote by \(\iota: H' \to H\) be the inclusion, which is a Hopf algebra morphism.  As we work over a field, one easily verifies that \((V, (V \otimes \iota) \rho)\) is an irregular partial \(H\)-comodule. 
\end{proof}

It is known that for any Lie algebra $\mathfrak g$, all partial modules of the Hopf algebra $U(\mathfrak g)$ are global (see \cite[Example 4.4]{ABV}). The following example shows that the case of quantum enveloping algebras is much more complicated.

\begin{example}\label{ex:quantum}
    Let \(q \neq \pm 1\) and consider \(U_q(\mathfrak{sl}_2)\). This Hopf algebra is generated by \(E, F, K\) and \(K^{-1}\) subject to relations
    \begin{gather*}
        KEK^{-1} = q^2 E, \quad KFK^{-1} = q^{-2} F, \quad EF-FE = \frac{K - K^{-1}}{q - q^{-1}} \\
        \Delta(E) = E \otimes 1 + K \otimes E, \quad \Delta(F) = F \otimes K^{-1} + 1 \otimes F, \quad \Delta(K) = K \otimes K \\
        \epsilon(E) = \epsilon(F) = 0, \quad \epsilon(K) = 1 \\
        S(E) = -K^{-1} E, \quad S(F) = -FK, \quad S(K) = K^{-1}.
    \end{gather*}
    If \(q^{2n} = 1\) (\(n > 1\)), then there is surjective Hopf algebra morphism $U_q(\mathfrak{sl}_2)\to T_n$, the Taft algebra. The Taft algebras are self-dual, so we find a Hopf algebra morphism $T_n \to U_q(\mathfrak{sl}_2)^\circ$. This map is injective (since it is the corestriction of the linear map \(T_n \cong T_n^* \to U_q(\mathfrak{sl}_2)^*\)). Remark that \(T_2 = H_4\), the Sweedler Hopf algebra, which is irregular by \cref{ex:irregular}. This shows that when \(q = i\), \(U_q(\mathfrak{sl}_2)^\circ\) is an example of an (infinite dimensional) irregular Hopf algebra. 
    In particular, there exist non-global partial comodules for \(U_q(\mathfrak{sl}_2)^\circ\) and in a similar way there exist (in contrast to the classical case) non-global partial modules of \(U_q(\mathfrak{sl}_2)\): one can precompose a partial comodule of \(H_4\) with the Hopf algebra projection \(U_q(\mathfrak{sl}_2) \to H_4\) to obtain a partial module with action
    \[U_q(\mathfrak{sl}_2) \otimes M \to H_4 \otimes M \to M.\]
\end{example}





\section{Partial representations of linear algebraic groups}
\label{se:algebriacgroups}

\subsection{Partial representations versus partial comodules over the coordinate algebra}

Let \(\Alg_k\) be the category of unital, associative, commutative \(k\)-algebras. 
Recall that a linear algebraic group is \(G\) is in fact a functor
	\[\mathbb{G} : \Alg_k \to \Grp\]
	such that the underlying \(k\)-functor
	\[\overline{\G} : \Alg_k \to \Set\]
	is represented by a finitely generated commutative \(k\)-algebra that is called the coordinate algebra and is denoted by \(\mathcal{O}(G)\). 
	
	Let \(V\) be a vector space and consider the functor
\[\End_V : \Alg_k \to \Mon\]
given on objects by 
\[\End_V(A) = \End_A(V \otimes A) \cong \Hom_k(V, V \otimes A)\]
and on morphisms by \(\End_V(f) : \Hom_k(V, V \otimes A) \to \Hom_k(V, V \otimes B) : \psi \mapsto (V \otimes f)\psi\). The corresponding \(\Set\)-valued functor is denoted by \(\overline{\End}_V\).

By the Yoneda Lemma, 
\begin{equation}
	\label{eq:yoneda}
	\Nat(\overline{\G}, \overline{\End}_V) \cong \overline{\End}_V(\mathcal{O}(G)) \cong \Hom_k(V, V \otimes \mathcal{O}(G)).
\end{equation}

\begin{definition}
	A natural transformation \(\alpha : \overline{\G} \Rightarrow \overline{\End}_V\) is a partial representation of \(G\) on \(V\) if for any \(A \in \Alg_k\) and \(x, y \in \G(A)\)
	\begin{enumerate}[(PR1)]
		\item \(\alpha_A(1_{\G(A)}) = {V \otimes A}\) (as elements in \(\End_A(V \otimes A)\));
		\item \(\alpha_A(x^{-1}) \alpha_A(x) \alpha_A(y) =  \alpha_A(x^{-1}) \alpha_A(xy)\);
		\item \(\alpha_A(x) \alpha_A(y) \alpha_A(y^{-1}) = \alpha_A(xy) \alpha_A(y^{-1})\).
	\end{enumerate}
	
\end{definition}

\begin{lemma}\label{le:parrepalg}
	A natural transformation \(\alpha : \overline{\G} \Rightarrow \End_V\) is partial representation of \(G\) if and only if the corresponding linear map under (\ref{eq:yoneda}) \(\rho_\alpha : V \to V \otimes \mathcal{O}(G)\) makes \(V\) into a partial \(\mathcal{O}(G)\)-comodule. 
\end{lemma}
\begin{proof}
    The check that \(\alpha\) satisfies (PR1) if and only if \(\rho_\alpha\) is counital is the same as in the global case.

For the partial associativity conditions, let us check that $\alpha$ satisfies (PR2) if and only if the corresponding linear map $\rho_{\alpha}$ satisfies
\begin{eqnarray}\label{pcomodcondition}
& \, & (V\otimes \mu \otimes \OG) (V\otimes \OG \otimes S\otimes \OG ) (\rho_{\alpha} \otimes \OG \otimes \OG) (\rho_{\alpha} \otimes \OG ) \rho_{\alpha} \nonumber \\
& = &(V\otimes \mu \otimes \OG) (V\otimes  \OG \otimes S\otimes \OG ) (V\otimes \OG \otimes \Delta_{\OG}) (\rho_{\alpha} \otimes \OG ) \rho_{\alpha} .
\end{eqnarray}
For every $A \in \Alg_k$ and every $x,y \in \G(A) \Hom_{\Alg_k} (\OG , A)$ we have
\begin{eqnarray*}
\alpha_A (x) \alpha_A (x^{-1}) \alpha_A (y) & = & (V\otimes \mu_A )(V\otimes \mu_A \otimes A) (V\otimes x \otimes A \otimes A) (\rho_{\alpha} \otimes A \otimes A) \\
& \,  &  (V \otimes x^{-1} \otimes A)  (\rho_{\alpha} \otimes A)  (V\otimes y) \rho_{\alpha} \\
& = & (V\otimes \mu_A )  (V\otimes \mu_A \otimes A) (V\otimes x \otimes x^{-1} \otimes y) \\
& \, &  (\rho_{\alpha} \otimes \OG \otimes \OG ) (\rho_{\alpha} \otimes \OG )  \rho_{\alpha} \\
& = & (V\otimes \mu_A )  (V\otimes \mu_A \otimes A) (V\otimes x \otimes x \otimes y)  (V \otimes \OG \otimes S \otimes \OG) \\
& \, &  (\rho_{\alpha} \otimes \OG \otimes \OG ) (\rho_{\alpha} \otimes \OG)  \rho_{\alpha} \\
& = & (V\otimes \mu_A )  (V\otimes x \otimes y)  (V\otimes \mu_{\OG} \otimes \OG)  (V \otimes \OG \otimes S \otimes \OG) \\
& \, &  (V \otimes \OG \otimes \Delta_{\OG} ) (\rho_{\alpha} \otimes \OG)  \rho_{\alpha} .
\end{eqnarray*}
On the other hand,
\begin{eqnarray*}
\alpha_A (x) \alpha_A (x^{-1}y) & = & (V\otimes \mu_A ) (V\otimes x \otimes A) (\rho_{\alpha} \otimes A)  (V\otimes x^{-1}y) \rho_{\alpha} \\
& = & (V\otimes \mu_A )  (V\otimes x \otimes x^{-1}y)  (\rho_{\alpha} \otimes \OG )  \rho_{\alpha} \\
& = & (V\otimes \mu_A )  (V\otimes A \otimes \mu_A ) (V\otimes x \otimes x^{-1} \otimes y)  (V\otimes \OG \otimes \Delta_{\OG}) \\
& \,  &  (\rho_{\alpha} \otimes \OG )  \rho_{\alpha} \\
& = & (V\otimes \mu_A )  (V\otimes \mu_A \otimes A) (V\otimes x \otimes x \otimes y)  (V \otimes \OG \otimes S \otimes \OG) \\
& \, &  (V \otimes \OG \otimes \Delta_{\OG} ) (\rho_{\alpha} \otimes \OG)  \rho_{\alpha} \\
& = & (V\otimes \mu_A )  (V\otimes x \otimes y)  (V\otimes \mu_{\OG} \otimes \OG)  (V \otimes \OG \otimes S \otimes \OG) \\
& \, &  (V \otimes \OG \otimes \Delta_{\OG} ) (\rho_{\alpha} \otimes \OG)  \rho_{\alpha} .
\end{eqnarray*}
Therefore (by replacing \(x\) by \(x^{-1}\) in the above calculation), $\alpha_A (x^{-1}) \alpha_A (x) \alpha_A (y)=\alpha_A (x^{-1}) \alpha_A (xy)$ for any $A \in \Alg_k$ if and only if
\begin{eqnarray*}
& \, & (V\otimes \mu_A )  (V\otimes x \otimes y)  (V\otimes \mu_{\OG} \otimes \OG)  (V \otimes \OG \otimes S \otimes \OG) \\
& \, &  (\rho_{\alpha} \otimes \OG \otimes \OG ) (\rho_{\alpha} \otimes \OG)  \rho_{\alpha} \\
& = & (V\otimes \mu_A )  (V\otimes x \otimes y)  (V\otimes \mu_{\OG} \otimes \OG)  (V \otimes \OG \otimes S \otimes \OG) \\
& \, &  (V \otimes \OG \otimes \Delta_{\OG} ) (\rho_{\alpha} \otimes \OG)  \rho_{\alpha} .
\end{eqnarray*}
To see that this implies the identity (\ref{pcomodcondition}), one can choose $A=\OG \otimes \OG$, $x=j_1 :\OG \rightarrow \OG\otimes \OG$ defined as $j_1 (f)=f\otimes 1$, and $y=j_2 :\OG \rightarrow \OG\otimes \OG$ defined as $j_2 (f)=1\otimes f$. 

Analogously, one can check that $\alpha_A (x)\alpha_A (y) \alpha_A (y^{-1})=\alpha_A (xy) \alpha_A (y^{-1})$ for any $A \in \Alg_k$ and any algebra morphisms $x,y:\OG \rightarrow A$ if and only if the corresponding linear map $\rho_{\alpha}$ satisfies
\begin{eqnarray*}
& \, & (V\otimes \OG \otimes \mu)  (V\otimes \OG \otimes \OG \otimes  S) (\rho_{\alpha} \otimes \OG \otimes \OG ) (\rho_{\alpha} \otimes \OG)  \rho_{\alpha} \\
& = & (V\otimes \OG \otimes \mu)  (V\otimes \OG \otimes \OG \otimes S) (V\otimes \Delta \otimes \OG ) (\rho_{\alpha} \otimes \OG)  \rho_{\alpha}
\end{eqnarray*}
Therefore, partial representations of a linear algebraic group $G$ are in one-to-one correspondence with partial comodules over the coordinate algebra $\OG$.
\end{proof}

    
\subsection{Connected linear algebraic groups}

From now until the end of \cref{se:algebriacgroups}, let \(k\) be an algebraically closed field of characteristic 0.

Suppose that \(G\) is connected. We will show that every partial representation of \(G\) is in fact global. To do this, we need a dual pairing between the coordinate algebra \(\OG\) and the universal enveloping algebra of its Lie algebra \(U(\mathfrak{g})\). 

By \cite[9.2.5]{montgomery}, 
\[U(\mathfrak{g}) \cong H' := \{f \in H^\circ \mid f((\ker \epsilon)^n) = 0 \text{ for some } n > 0\}.\]
Moreover, for a finitely generated \(k\)-algebra \(H\), \(H'\) is dense in \(H^*\) if and only if \(\bigcap_{n \geq 0} (\ker \epsilon)^n = 0\). 

By \cite[4.6.4]{abe}, the coordinate algebra of a connected affine algebraic group satisfies this property. This implies the existence of a nondegenerate dual pairing between \(U(\mathfrak{g})\) and \(\OG\). Explicitely, it is given by 
\[\langle \cdot, \cdot \rangle : U(\mathfrak{g}) \otimes \OG \to k : \langle d_1 \cdots d_n, x \rangle = \epsilon(D_1 \cdots D_n(x))\]
where \(D_1, \dots, D_n \in \mathfrak{g}\) are left-invariant derivations on \(\OG\) and \(d_1, \dots, d_n\) are the corresponding generators of \(U(\mathfrak{g})\). One can check that this is a non-degenerate Hopf algebra pairing. 

Following the proof of \cite[Theorem 4.14]{coreps}, \((M, \rho)\) is a partial \(\OG\)-comodule if and only if the corresponding map \[\lambda : U(\mathfrak{g}) \otimes M \to M : d \otimes m \mapsto m^{(0)} \langle d, m^{(1)} \rangle\]
makes \(M\) into a partial \(U(\mathfrak{g})\)-module. But every partial \(U(\mathfrak{g})\)-module is global (see \cite[4.4]{ABV}) hence, by non-degeneracy of the pairing, \((M, \rho)\) is a global \(\OG\)-module.
We have shown the following theorem.
\begin{theorem}
    \label{th:connectedgroup}
    Let \(k\) be an algebraically closed field of characteristic \(0\) and let \(G\) be a linear algebraic \(k\)-group. If \(G\) is connected, then every partial representation of \(G\) is global.
\end{theorem}

\begin{example}
    Consider the additive group, i.~e.~the \(k\)-group functor which associates the group \((A, +)\) to each commutative \(k\)-algebra \(A\). Its coordinate algebra is \(k[t]\), where \(t\) is a primitive element; \(\Delta(t) = 1 \otimes t + t \otimes 1\). This algebra is a domain, hence the additive is connected. From the discussion above, we can conclude that every partial comodule over \(k[t]\) is global. Remark that \(k[t]\) is exactly the enveloping algebra of the trivial Lie algebra \(k\). We already knew that this Hopf algebra does not have true partial \textit{modules} either. 
\end{example}

\begin{example}
    Consider the multiplicative group, i. e.
    \[\G(A) = (A^\times, \cdot)\]
    for any commutative \(k\)-algebra \(A\). Then \(\OG\) is the algebra of Laurent polynomials \(k[t, t^{-1}]\), where \(t\) is grouplike. This algebra is again a domain, hence \(G\) is connected and \(k[t, t^{-1}]\) does not admit partial comodules other than global ones. Remark that \(\OG\) is isomorphic (as a Hopf algebra) to the group algebra \(k\Z\).
\end{example}

\cref{th:connectedgroup} has an easy corollary:
\begin{corollary}
     Let \(k\) be an algebraically closed field of characteristic \(0\) and let \(G\) be a linear algebraic \(k\)-group. If \(G\) is connected, then \(\OG\) is regular.
\end{corollary}

In fact, we expect \(\OG\) to be regular for any linear algebraic group \(G\) over \(k\). We will not prove this in detail in this article, but we will motivate this claim and give some concrete examples. 

\begin{remark} \label{inductionofpartialreps} For any linear algebraic group \(\G\), its connected component \(\G^\circ\) is a connected normal subgroup, and the corresponding quotient is finite (i.~e.~its coordinate algebra is finite dimensional). The article \cite{DHSV} studies partial representations of a finite group \(G\) for which the restriction to a subgroup \(H\) is a global representation. Irreducible \(H\)-global \(G\)-partial representations are constructed by taking a subset \(A\) of \(G/H\) that contains \(H\). A (global) irreducible representation \(V\) of the subgroup \(K = \{g \in G \mid gA = A\}\) induces a representation \(kG \otimes_{kK} V\) of \(G\), which can be restricted to an irreducible \(H\)-global \(G\)-partial representation
\[\bigoplus_{g_iK \in A/K} V^{g_i} \subseteq \bigoplus_{g_iK \in G/K} V^{g_i} = kG \otimes_{kK} V.\]
Here, each $V^{g_i}$ is isomorphic to $V$ and is equal to $k(g_i K) \otimes_{kK} V$. The isomorphism $\phi_{i} :V\rightarrow V^{g_i}$ can be written explicitely as $\phi_i (v)=g_i \otimes_{kK} v$.
Denoting by $\pi :K\rightarrow \mathsf{GL}(V)$ the global representation of $K$, we have the partial representation of $G$, $\widetilde{\pi}:G\rightarrow \text{End}_k \left( \bigoplus_{g_iK \in A/K} V^{g_i} \right)$, given by
\[
\widetilde{\pi}(g) (\phi_i (v))=\left\{  \begin{array}{lcl} \phi_j (\pi (k)(v)) & \text{ if } & gg_i =g_j k, \quad \text{ for } k\in K, \text{ and } g_j K\in A/K \\
0 & & \text{otherwise}  ,\end{array}\right.
\]
in which $A/K$  denotes the set of left cosets of $K$ contained in $A$. Remark that if \(H \trianglelefteq G\), then \(H \leq K \leq G\). Since the techniques in \cite{DHSV} depend on the finiteness of the quotient \(G/H\) and less on the size of \(G\) itself, we can expect that partial representations of a linear algebraic group \(G\) can be built out of global representations of a subgroup that contains \(G^\circ\) and the action of \(G\) on the finite set \(\mathcal{P}(G/G^\circ)\). 
\end{remark}

\subsection{A trivial partial representation of linear algebraic groups} \label{subs:trivial} For a linear algebraic group \(G\), the group of connected components \(G/G^\circ\) is finite and has coordinate algebra \(\pi_0(G)\) which is a Hopf subalgebra of \(\OG =: H\). It is separable and hence as an algebra isomorphic to \(k^{n + 1}\) for some \(n\). Denote by \(\{e_0, \dots, e_n\}\) a complete system of orthogonal idempotents and let \(e_0\) be the idempotent that is sent to \(1\) under the counit.

Now (see \cite[Section 6.7]{Waterhouse}) \(\mathcal{O}(G^\circ) =: H_+\) is isomorphic to \(e_0 \OG\) and the projection 
\[\pi : H \to H_+ : h \mapsto e_0 h\]
is a Hopf morphism because its kernel is the Hopf ideal \((\ker \varepsilon_{\pi_0(G)}) H\). This ideal is in fact \(e_1 H \times \cdots \times e_n H\), let us denote it by \(H_-\). The section 
\[\iota : H_+ \to H\] 
is multiplicative but it is an algebra nor a coalgebra morphism. By counitality in \(\pi_0(G)\), we have that
\[\Delta_H(e_0) \in e_0 \otimes e_0 + H_- \otimes H_-\]
and it follows that 
\[\Delta_H(H_+) \subseteq H_+ \otimes H_+ + H_- \otimes H_-.\]

Since \(\pi_0(G)\) is a finite dimensional commutative semisimple Hopf algebra, it is isomorphic to \(k(G/G^\circ)^*\) and \(e_0, \dots, e_n\) is exactly the basis dual to \(G/G^\circ\); hence \(S(e_0) = e_0\). 

\begin{lemma}
Let \((M, \rho)\) be a comodule over \(H_+\) and consider
\[\tilde{\rho} = (M \otimes \iota) \rho : M \to M \otimes H_+ \to M \otimes H.\]
Then \((M, \tilde{\rho})\) is a partial comodule over \(H\). 
\end{lemma}
\begin{proof}
For the counitality, remark that \(\epsilon_H(h) = \epsilon_{H_+}(\pi(h))\) for every \(h \in H\), hence \(\epsilon_H(\iota(h')) = \epsilon_{H_+}(h')\) for all \(h' \in H_+\). For \ref{PCM2}, 
\begin{gather*}
    (M\otimes H \otimes \mu)(M\otimes H \otimes H \otimes S)(M\otimes \Delta_H \otimes H) (\tilde\rho \otimes H )\tilde\rho(m) \\ = m^{(0)(0)} \otimes (e_0)_{(1)} {m^{(0)(1)}}_{(1)}  \otimes (e_0)_{(2)} {m^{(0)(1)}}_{(2)}   S(e_0 m^{(1)}) \\
    = m^{(0)(0)} \otimes (e_0)_{(1)} {m^{(0)(1)}}_{(1)}  \otimes e_0 (e_0)_{(2)} {m^{(0)(1)}}_{(2)}   S(m^{(1)}) \\
    = m^{(0)(0)} \otimes e_0 {m^{(0)(1)}}_{(1)}  \otimes e_0 {m^{(0)(1)}}_{(2)}  S(m^{(1)}) \\
    = m^{(0)(0)} \otimes \pi( {m^{(0)(1)}}_{(1)})  \otimes \pi( {m^{(0)(1)}}_{(2)})  S(m^{(1)}).
\end{gather*}
Remark now that \(\Delta_{H_+}(\pi(h)) = \pi(h_{(1)}) \otimes \pi(h_{(2)})\) for all \(h \in H\) because \(\pi\) is a Hopf morphism. The elements \(m^{(0)(1)}\) are in \(H_+\), hence 
\begin{gather*}
    m^{(0)(0)} \otimes \pi( {m^{(0)(1)}}_{(1)})  \otimes \pi( {m^{(0)(1)}}_{(2)})  S(m^{(1)}) \\ = (M \otimes H \otimes \mu)(M \otimes \iota \otimes \iota \otimes S)(M \otimes \Delta_{H_+} \otimes H)(\rho \otimes H) \tilde \rho(m).
\end{gather*}
But \(\rho\) defines an \(H_+\) comodule structure on \(M\), hence the last expression is equal to
\begin{gather*}
    (M \otimes H \otimes \mu)(M \otimes \iota \otimes \iota \otimes S)(\rho \otimes H_+ \otimes H)(\rho \otimes H) \tilde \rho(m) \\
    = (M\otimes H \otimes \mu)(M\otimes H \otimes H \otimes S)(\tilde \rho \otimes H \otimes H) (\tilde \rho \otimes H ) \tilde \rho(m).
\end{gather*}

The proof of \ref{PCM3} is similar.
\end{proof}

\subsection{Example: the orthogonal group}\label{se:orthogonal}
Consider \(G = \mathsf{O}(n)\), the orthogonal group. Then \(G^\circ\) is \(\mathsf{SO}(n)\) and \(G/G^\circ\) has two elements (\(\pi_0(G)\) is two-dimensional). To get some intuition with this group, let us describe explicitly the coordinate algebra when \(n = 2\). 

\[\OG = k[a, b, c, d]/(a^2 + b^2 - 1, c^2 + d^2 - 1, ac + bd)\]

and
\[\mathcal{O}(G^\circ) = k[A, B]/(A^2 + B^2 - 1)\]

The quotient map is given by
\[\pi_0 : \OG \to \mathcal{O}(G^\circ) : \begin{cases}
 a \mapsto A, \\
 b \mapsto B, \\
 c \mapsto -B, \\
 d \mapsto A
\end{cases}
\]
and its multiplicative section by
\[\iota : \mathcal{O}(G^\circ) \to \OG : \begin{cases}
A \mapsto \frac{a + d}{2}, \\
B \mapsto \frac{b - c}{2}
\end{cases}
\]
The idempotents are
\begin{align*}
    e_0 &= \iota(1) = \iota(A^2 + B^2) = \frac{1 + ad - bc}{2} = \frac{1 + D}{2}, \\
    e_1 &= 1 - e_0 = \frac{1 - D}{2}
\end{align*}
where \(D = ad - bc\) is the determinant. 

The aim of this subsection is to show that any partial comodule over \(\OG\) is the direct sum of a global comodule over \(\OG\) and a partial comodule of the form described in Subsection \ref{subs:trivial}. 

Let \((V, \rho)\) be any partial comodule over \(\OG\). For an algebra map \(x : \OG \to k\) (i.~e. \(x \in \mathsf{O}(n, k)\)) we define
\begin{equation}
    \label{eq:P}
    P_x : V \to V : v \mapsto (V \otimes x \otimes x)(V \otimes S \otimes H)(\rho \otimes H)\rho (v).
\end{equation}
We claim that \(P_x\) is a projection. Indeed, for any \(v \in V\)
\begin{align*}
    P_x^2(v) &= P_x(v^{(0)(0)} x(S(v^{(0)(1)}) v^{(1)})) \\
    &= v^{(0)(0)(0)(0)} x(S(v^{(0)(0)(0)(1)}) v^{(0)(0)(1)} S(v^{(0)(1)}) v^{(1)}) \\
    \overset{\ref{PCM2}}&{=} v^{(0)(0)(0)} x(S({v^{(0)(0)(1)}}_{(1)}) {v^{(0)(0)(1)}}_{(2)} S(v^{(0)(1)}) v^{(1)}) \\
    &= v^{(0)(0)(0)} \epsilon(v^{(0)(0)(1)}) x(S(v^{(0)(1)}) v^{(1)})) \\
    \overset{\ref{PCM1}}&{=} v^{(0)(0)} x(S(v^{(0)(1)}) v^{(1)}) \\
    &= P_x(v).
\end{align*}

An algebra morphism \(x : \OG \to k\) sends either \(e_0\) to \(1\) and \(e_1\) to \(0\) or \(e_0\) to \(0\) and \(e_1\) to \(1\). In the first case \(x \iota \pi = x\) (and \(x \in \mathsf{SO}(n, k)\)), in the second case  \(x \iota \pi = 0\) (and \(x \notin \mathsf{SO}(n, k)\)). 

Let us write again \(H = \OG\), \(H_+ = \mathcal{O}(G^\circ) = e_0H\) and \(H_- = e_1H\), so that \(H = H_+ \times H_-\). 

\begin{lemma}
    \label{le:globality}
    If \(x \in \mathsf{SO}(n, k)\) and \(y \in \mathsf{O}(n, k)\), then
    \[(V \otimes x \otimes y)(\rho \otimes H)\rho = (V \otimes x \otimes y)(V \otimes \Delta_H)\rho.\]
    If \(x \in \mathsf{SO}(n, k)\), then \(P_x = V\). If \(y, y' \in \mathsf{O}(n, k) \setminus \mathsf{SO}(n, k)\), then \(P_y = P_{y'}\).
\end{lemma}
\begin{proof}
    This is a reformulation of the fact that when \(x \in \mathsf{SO}(n, k)\), then \(\alpha_k(x)\) is invertible in \(\End_V(k)\) (because any partial representation of \(\mathsf{O}(n)\) is global on \(\mathsf{SO}(n)\) with \(\alpha_k(x)^{-1} = \alpha_k(x^{-1})\)). We have
    \[\alpha_k(x) \alpha_k(y) = \alpha_k(xy)\]
    by (PR2). In particular, \(\alpha_k(x^{-1}) \alpha_k(x) = V\). If \(y, y' \in \mathsf{O}(n, k) \setminus \mathsf{SO}(n, k)\), then \((y \otimes y'S)\Delta_H \in \mathsf{SO}(n, k)\) hence \(\alpha_k(yy'^{-1})\) is invertible. It follows that
    \begin{align*}
        \alpha_k(y^{-1}) \alpha_k(y) &= \alpha_k(y^{-1}) \alpha_k(yy'^{-1}) \alpha_k(y'y^{-1}) \alpha_k(y) \\ &= \alpha_k(y'^{-1}) \alpha_k(y'y^{-1}) \alpha_k(y) = \alpha_k(y'^{-1}) \alpha_k(y'y^{-1}) \alpha_k(yy'^{-1}y')\\
        &= \alpha_k(y'^{-1}) \alpha_k(y'y^{-1}) \alpha_k(yy'^{-1}) \alpha_k(y') = \alpha_k(y'^{-1}) \alpha_k(y'). \qedhere
    \end{align*} 
\end{proof}

\begin{corollary}
    \label{cor:centrality}
    For any \(y \in \mathsf{O}(n, k) \setminus \mathsf{SO}(n, k)\) and \(z \in \mathsf{O}(n, k)\),
    \[(V \otimes yS \otimes y \otimes z)\rho^3 = (V \otimes z \otimes yS \otimes y)\rho^3.\]
\end{corollary}
\begin{proof}
    Since for \(z' = zS\) either \(P_{z'} = V\) or \(P_{z'} = P_y\), we have
    \begin{align*}
        (V \otimes yS \otimes y \otimes z)\rho^3 &= (V \otimes z \otimes z \otimes y \otimes y \otimes z)(V \otimes H \otimes S \otimes S \otimes H)\rho^5 \\
        \overset{\ref{PCM5}}&{=} (V \otimes z \otimes z \otimes y \otimes y \otimes z)(V \otimes H \otimes S \otimes S \otimes \Delta_H) \rho^4 \\
        \overset{\ref{PCM3}}&{=} (V \otimes z \otimes z \otimes y \otimes y \otimes z)(V \otimes H \otimes S \otimes S \otimes H \otimes H)(V \otimes \Delta_H \otimes \Delta_H) \rho^3.
    \end{align*}
    This last expression is exactly \((V \otimes z) \rho P_u\) where \(u = (y \otimes z) \Delta_H\). If \(z \in \mathsf{SO}(2, k)\), then \(u \notin \mathsf{SO}(2, k)\) and \(P_u = P_y\) by \cref{le:globality}. In the other case, \(P_u = V\), but since \[(V \otimes z)\rho = (V \otimes z)\rho P_z = (V \otimes z)\rho P_y,\] we are done in both cases.
\end{proof}

 Fix now \(y \in \mathsf{O}(n, k) \setminus \mathsf{SO}(n, k)\) and write \(P = P_y\). 
Write \(V_g = \mathrm{im}\, P\) and \(V_p = \ker P\), obviously \(V = V_g \oplus V_p\). 
We will show that \(V_g\) is a global \(\OG\)-comodule and that \(V_p\) is a global \(\mathcal{O}(G^\circ)\)-comodule extended to \(\OG\) to create a partial \(\OG\)-module, as described in Subsection \ref{subs:trivial}.

\begin{lemma}
    The partial coaction \(\rho : V \to V \otimes H\) restricts to a map \(V_g \to V_g \otimes H\) and \((V_g, \rho_{|V_g})\) is a global \(H\)-comodule. 
\end{lemma}
\begin{proof}
    First we show that \(\rho\) restricts to a map \(V_g \to V_g \otimes H\). Take \(v \in V_g\). It suffices to show that
    \begin{equation}
    	(V \otimes yS \otimes y \otimes H)(\rho \otimes H \otimes H)(\rho \otimes H)\rho(v) = \rho(v). \label{eq:globalcom}
    \end{equation}
    Recall that 
    \[\rho(v) = \rho(P_y(v)) = (V \otimes H \otimes yS \otimes y)(\rho \otimes H \otimes H)(\rho \otimes H)\rho(v).\] 
    Since \(\mathsf{O}(n, k)\) separates the points of \(\OG\) (recall that \(k\) is algebraically closed and of characteristic \(0\)), (\ref{eq:globalcom}) follows from \cref{cor:centrality}.

    Finally, let us show that \(V_g\) is an \(H\)-comodule. We will show that for any \(z, z' \in \mathsf{O}(n,k)\) and \(v \in V_g\)
    \[(V \otimes z \otimes z')(\rho \otimes H)\rho(v) = (V \otimes z \otimes z')(V \otimes \Delta_H)\rho(v).\]
    Since \(P_{z'}(v) = v\) (by \cref{le:globality}, either \(P_{z'} = V\) or \(P_{z'}(v) = P_y(v) = v\)), 
    \begin{align*}
        (V \otimes z \otimes z')(\rho \otimes H)\rho(v) &= (V \otimes z \otimes z' \otimes z'S \otimes z')\rho^4(v) \\
        \overset{\ref{PCM2}}&{=} (V \otimes z \otimes z' \otimes z'S \otimes z')(V \otimes \Delta_H \otimes H \otimes H)\rho^3(v) \\
        &= (V \otimes z \otimes z')(V \otimes \Delta_H)\rho P_{z'}(v) \\
        &= (V \otimes z \otimes z')(V \otimes \Delta_H)\rho(v). \qedhere
    \end{align*}
\end{proof}

\begin{lemma}
    The partial coaction \(\rho : V \to V \otimes H\) restricts to a map \(V_p \to V_p \otimes H_+\) and \((V_p, \rho_{|V_p})\) is a global \(H_+\)-comodule. 
\end{lemma}
\begin{proof}
    We show first that \(\rho\) restricts to a map \(V_p \to V_p \otimes H\). It suffices to show that for any \(v \in V_p\) and any \(z \in \mathsf{O}(n, k),\) 
    \[(V \otimes yS \otimes y \otimes z)\rho^3(v) = 0.\]
    This is true by \cref{cor:centrality}, since \(P_y(v) = 0\). Next, remark that for any \(z \in \mathsf{O}(n, k) \setminus \mathsf{SO}(n, k)\)
    \begin{align*}
        (V \otimes z) \rho(v) &= (V \otimes z \otimes zS \otimes z) (V \otimes \Delta_H \otimes H) \rho^2(v) \\
        &= (V \otimes z \otimes zS \otimes z)\rho^3(v) \\
        &= (V \otimes z) \rho P_z(v) = 0.
    \end{align*}
    It follows that indeed \(\rho(V_p) \subseteq V_p \otimes H_+\). This means that \(\rho_{|V_p} = (V \otimes \pi)\rho\), and since \(\pi\) is a Hopf morphism, \((V_p, \rho_{|V_p})\) is a partial \(H_+\)-comodule which is global because \(H_+ = \mathcal{O}(G^\circ)\) (\cref{th:connectedgroup}). 
\end{proof}

\subsection{Example: a group of monomial matrices}\label{se:monomial}
\begin{example}\label{permutationmatrices}
    Consider the algebraic group $G$ defined by (\(A \in \Alg_k\))
    \begin{align}\label{grupoG}
    \G(A) & = \left\{ \left. \left(\begin{array}{ccc} a & 0 & 0\\ 0 & b & 0 \\ 0 & 0 & c \end{array} \right) \in M_3 (A) \; \right| \; a,b,c\in A^{\times} \right\} \cup \left\{ \left. \left(\begin{array}{ccc} 0 & 0 & c\\ a & 0 & 0 \\ 0 & b & 0 \end{array} \right) \in M_3 (A) \; \right| \; a,b,c\in A^{\times} \right\} \nonumber \\
    & \cup \left\{ \left. \left(\begin{array}{ccc} 0 & b & 0\\  0 & 0 & c \\ a & 0 & 0 \end{array} \right) \in M_3 (A) \; \right| \; a,b,c\in A^{\times} \right\} .
    \end{align}
    Since we work over an algebraically closed field in characteristic \(0\), we will concentrate on the group \(G = \G(k)\). 
    The first subset in (\ref{grupoG}) is the connected component of the identity, which is an algebraic subgroup, denoted by $G^{\circ}$. It is easy to see that $G/G^{\circ} \cong C_3$, the cyclic group of order $3$ and that the second and third components of $G$ are, respectively, the left cosets $gG^{\circ}$ and $g^2 G^{\circ}$, in which 
    \begin{equation}\label{matrizg}
    g=\left( \begin{array}{ccc} 0 & 0 & 1 \\ 1 & 0 & 0 \\ 0 & 1 & 0 \end{array} \right)
    \end{equation}
    is the permutation matrix corresponding to the generator of the cyclic group $C_3$. We are going to construct nontrivial partial comodules of $\OG$ from global representations of the connected subgroup $G^{\circ}$ using partial representations of $C_3$ (other than the trivial partial representations described in Subsection \ref{subs:trivial}).  We will follow the procedure given in \cite{DHSV}, as explained in Remark \ref{inductionofpartialreps}. 

    Consider, for example, the three dimensional representation $\pi: G^\circ \rightarrow \text{End}_{k^3}$ given by
    \[
    \pi (\sigma)v_1 =\alpha v_1 , \qquad \pi (\sigma )v_2 =\beta v_2 ,\qquad  \pi (\sigma )v_3 =\gamma v_3 ,
    \]
    for 
    \begin{equation}\label{matrizsigma}
    \sigma = \left( \begin{array}{ccc} \alpha & 0 & 0\\ 0 & \beta & 0 \\ 0 & 0 & \gamma \end{array} \right).
    \end{equation}
    Take the set $\{ G^\circ, gG^\circ \} \subseteq \mathcal{P}_{G^\circ} (G/G^\circ )$, in which $\mathcal{P}_{G^\circ} (G/G^\circ )$ denotes the set of subsets of $G/G^\circ$ which contain $G^\circ$ and $g$ is the permutation matrix described in (\ref{matrizg}). In this case, both subgroups $H$ and $K$, described in the method of partial induction, coincide with $G^\circ$. Now, take the vector space $W=V\oplus V^g$ and denote the basis of $W$ by $\{ v_1 , v_2 , v_3 , w_1 , w_2 , w_3 \}$ in which $\{ v_1, v_2, v_3 \}$ is the basis vectors of $V$ and $ \{w_1 , w_2 , w_3 \}$ is a basis of $V^g$ coming from the isomorphism $\phi_g :V\rightarrow V^g$. Using the procedure given in Remark \ref{inductionofpartialreps}, we have the partial representation $\widetilde{\pi}:G\rightarrow \text{End}_k (W)$ given as follows:
    \begin{equation}\label{representacaoW}
    \begin{array}{lll} \widetilde{\pi} (\sigma)(v_1) =\alpha v_1 , \quad & \widetilde{\pi} (g\sigma)(v_1) =\alpha w_1 , \quad & \widetilde{\pi} (g^2 \sigma)(v_1) =0 \\
    \widetilde{\pi} (\sigma)(v_2) =\beta v_2 , \quad & \widetilde{\pi} (g\sigma)(v_2) =\beta w_2 , \quad & \widetilde{\pi} (g^2 \sigma)(v_2) =0 \\
    \widetilde{\pi} (\sigma)(v_3) =\gamma v_3 , \quad & \widetilde{\pi} (g\sigma)(v_3) =\gamma w_3 , \quad & \widetilde{\pi} (g^2 \sigma)(v_3) =0 \\
    \widetilde{\pi} (\sigma)(w_1) =\beta w_1 , \quad & \widetilde{\pi} (g\sigma)(w_1) =0 , \quad & \widetilde{\pi} (g^2 \sigma)(w_1) =\beta v_1 \\
    \widetilde{\pi} (\sigma)(w_2) =\gamma w_2 , \quad & \widetilde{\pi} (g\sigma)(w_2) =0 , \quad & \widetilde{\pi} (g^2 \sigma)(w_2) =\gamma v_2 \\
    \widetilde{\pi} (\sigma)(w_3) =\alpha w_3 , \quad & \widetilde{\pi} (g\sigma)(w_3) =0 , \quad & \widetilde{\pi} (g^2 \sigma)(w_3) =\alpha v_3 ,
    \end{array}
    \end{equation}
    in which the matrix $\sigma$ is given by the expression (\ref{matrizsigma}).
    
    Now we can use these induced partial representations to create nontrivial examples of partial comodules over $\OG$. The algebra $\OG$ is the quotient of the polynomial algebra
    \[
    k[a_1 , a_2 , a_3 , b_1 , b_2 , b_3 , c_1 , c_2 , c_3 , t ] ,
    \]
    in which the variables $a_i$, $b_i$ and $c_i$ correspond, respectively, to the entries in the first, second and third columns, by the ideal 
    \begin{eqnarray*}
    \mathcal{I} & = & \langle a_1 a_2 \, , \,  a_1 a_3 \, ,\,  a_2 a_3 \, ,\,  b_1 b_2 \, ,\,  b_1 b_3 \, ,\,  b_2 b_3 \, ,\,  c_1 c_2 \, ,\,  c_1 c_3 \, ,\,  c_2 c_3 ,\\
    & \,  & a_1 b_1 \, ,\, a_1 c_1 \, ,\, b_1 c_1 \, ,\, a_2 b_2 \, ,\, a_2 c_2 \, ,\,  b_2 c_2 \, ,\, a_3 b_3 \, ,\, a_3 c_3 \, ,\, b_3 c_3 \, , \\  & \, & a_1 b_3   \, ,\, a_2 b_1 \, ,\, a_3 b_2  \, ,\,
      t(a_1 b_2 c_3 +a_2 b_3 c_1 +a_3 b_1 c_2 ) -1 \rangle .
    \end{eqnarray*}
    The relations determined in the quotient just tell us that the matrices should have only one invertible entry in each row and each column, that the matrices should have the form given in (\ref{grupoG}) and the their determinant should be invertible. The algebra $\mathcal{O}(G^\circ)$ is given by
    \[
    k[a, b, c, D ]/\langle abcD -1 \rangle ,
    \]
    which is simply the algebra of Laurent polynomials in three variables. Both are Hopf algebras, being the Hopf structure of $\mathcal{O}(G^{\circ})$ given simply by making the generators $a$, $b$, $c$ and $D$ grouplike. The Hopf algebra structure of $\OG$ is given by viewing the group $G$ as a subgroup of the matrix group $\mathsf{GL}_3$.
    As the group has three distinct connected components, there are three idempotent elements, $e_0$, $e_1$ and $e_2$ such that the algebra $\OG$ can be written as 
    \[
    \OG =e_0 \OG \times e_1 \OG \times e_2 \OG ,
    \]
    where $e_0 \OG \cong \mathcal{O}(G^{\circ})$. More explicitly, the idempotents are
    \[
    e_0 =a_1 b_2 c_3 t, \qquad e_1 =a_2 b_3 c_1 t, \qquad e_2 =a_3 b_1 c_2 t.
    \]
    There are three different multiplicative inclusions $\iota_i :\mathcal{O} (G^{\circ}) \rightarrow \OG$, for $i=0,1,2$, given by
    \[
    \begin{array}{llll}
    \iota_0 (a)=a_1 , \quad & \iota_1 (a)=a_2 , \quad & \iota_2 (a)=a_3 , \\
    \iota_0 (b)=b_2 , \quad & \iota_1 (b)=b_3 , \quad & \iota_2 (b )=b_1 , \\
    \iota_0 (c)=c_3 , \quad & \iota_1 (c)=c_1 , \quad & \iota_2 (c )=c_2  ,\\
    \iota_0 (D)=t , \quad & \iota_1 (D)=t , \quad & \iota_2 (D)=t  .
    \end{array}
    \]
    Basically, each one of the inclusions the inclusions $\iota_i$ throws,  respectively, the algebra $\mathcal{O} (G^{\circ})$ into the ideals $e_i \OG$.
    
    Finally, using the partial representation (\ref{representacaoW}) one can endow the vector space $W=V\oplus V^g$ with a structure of partial $\OG$-comodule. first note that $V$ is a global $\mathcal{O}(G^\circ)$-comodule, with structural map $\rho :V\rightarrow V\otimes \mathcal{O}(G^\circ)$ written as
    \[
    \rho (v_1)=v_1 \otimes a, \qquad \rho (v_2 )=v_2 \otimes b , \qquad \rho (v_3 )=v_3 \otimes c .
    \]
    Using the inclusions $\iota_0$, $\iota_1$ and $\iota_2$ previously described and the representation (\ref{representacaoW}), we construct the structural map $\widetilde{\rho}: W\rightarrow W\otimes \OG$,
    \[
    \begin{array}{lll}
    \widetilde{\rho} (v_1)=v_1\otimes a_1 +w_1 \otimes a_2 , \quad & \widetilde{\rho} (v_2)=v_2\otimes b_2 +w_2 \otimes b_3 , \quad &
    \widetilde{\rho} (v_3)=v_3\otimes c_3 +w_3 \otimes c_1 , \\
    \widetilde{\rho} (w_1)=v_1\otimes b_1 +w_1 \otimes b_2 , \quad &
    \widetilde{\rho} (w_2)=v_2\otimes c_2 +w_2 \otimes c_3 , \quad &
    \widetilde{\rho} (w_3)=v_3\otimes a_3 +w_3 \otimes a_1 , \end{array}
    \]
    which makes $W$ into a partial $\OG$-comodule.
\end{example}

\section{Explicit construction of the right adjoint of the forgetful functor}
\selabel{explicit}

\subsection{Topological vector spaces}\label{se:tvs}


In this subsection we collect some known definitions and results on topological vector spaces and topological coalgebras. More details and proofs can be found in \cite{takeuchi}, although we follow a slightly different approach.

%
%
Let us first introduce some notation. 
For any set $S$, we denote by $P_f(S)$ the finite powerset of $S$, that is, the set of all finite subsets of $S$. 
Let $V$ be a vector space, $I$ a (finite) set and $f_i:V\to V_i$, $i\in I$ a collection of  surjective $k$-linear maps.
Then we denote by $V_I$ the vector space $V/(\bigcap_{i\in I} \ker f_i)$ and by $f_I:V\to V_I$ the canonical projection. Throughout we consider a fixed base field $k$ endowed with the discrete topology.

\begin{definition}
We define the category of {\em topological $k$-vector spaces}, denoted as $\TVS_k$ as the category where
\begin{itemize}
\item an \ul{object} is a family $\ul V=(v_i:V\to V_i)_{i\in V_*}$, where $V_*$ is a set, $V$ and $V_i$ are vector spaces and $v_i$ is a jointly monic family of surjective linear maps.
\item a \ul{morphism} $\ul f:\ul V=(v_i:V\to V_i)_{i\in V_*}\to \ul W=(w_j:W\to W_j)_{j\in W_*}$ is a $k$-linear map $f:V\to W$ for which there exist a map $f_*:W_*\to P_f(V_*)$ and a family of $k$-linear maps $(f_j:V_{f(j)}\to W_j)_{j\in W_*}$ such that $f_j\circ v_{f_*(j)} = w_j \circ f$ for all $j\in W_*$.
\end{itemize}
\[
\begin{tikzcd}
	P_f(V_*) && W_* \arrow{ll}[swap]{f_*}\\[-20pt]
	V \arrow{rr}{f} \arrow{d}[swap]{v_{f_*(j)}} && W \arrow{d}{w_j} \\
	V_{f_*(j)} \arrow{rr}{f_j} && W_j
\end{tikzcd}
\]

We will often identify a topological vector space $\ul V$ with its underlying vector space $V$, and a morphism $\ul f$ with the underlying $f$.
\end{definition}

\begin{remarks}
We like to think of a morphism $\ul f:\ul V=(v_i:V\to V_i)_{i\in V_*}\to \ul W=(w_j:W\to W_j)_{j\in W_*}$ as a triple $\ul f=(f_*,f,f_j)$. However, a morphism is completely determined by the linear map $f$, as the existence of the maps $f_*$ and $f_j$ is considered as a property and not part of the structure. Remark also two triples $\ul f=(f_*,f,f_j)$ and $\ul g=(g_*,g,g_j)$ define the same morphism in $\TVS_k$ if one of the following equivalent conditions hold:
\begin{itemize}
\item $f=g$;
\item $f_j\circ v_{f_*(j)}=g_j\circ v_{g_*(j)}$ for all $j\in W_*$.
\end{itemize}
The equivalence follows easily from the fact that the morphisms $w_j$ are jointly monic. 

Any topological vector space $\ul V=(v_i:V\to V_i)_{i\in V_*}$ is isomorphic in $\TVS_k$ with the topological vector space $(v_I:V\to V_I)_{I\in P_f(V_*)}$. In other words, we can replace the index set of a topological vector space by its finite power set, which is moreover a directed partially ordered set by the inclusion. As follows from the proof of the next proposition, replacing $V_*$ by $P_f(V_*)$ is equivalent to replacing a subbasis for the topology of $V$ with a base for the topology. One could choose to consider the directed partially ordering on the index set as part of the structure of a topological vector space, however we prefer not to do this for sake of simplicity. 
\end{remarks}

Let us observe that this definition is equivalent to the one given in \cite{takeuchi}.

\begin{proposition}
The category $\TVS_k$ of topological vector spaces is equivalent to the category whose objects are vector spaces endowed with a translation invariant Hausdorff topology such that $0$ has a subbasis of neighborhoods consisting of open linear subspaces and whose morphisms are continuous linear maps. 
\end{proposition}

\begin{proof}
On the level of objects, this equivalence is obtained by endowing for any object $\ul V$ in $\TVS_k$ the vector space $V$ with the coarsest topology for which all maps $v_i$ are continuous and all spaces $V_i$ are discrete. More precisely, a subbasis of linear open neighborhoods for $0$ is given by the kernels of the maps $v_i$ (hence a base of linear open neighborhoods is given by finite intersections of those kernels). By requiring that the topology is translation invariant, this base defines a unique topology on $V$. The fact that the family $v_i$ is jointly monic implies that the intersection of all $\ker v_i$ is the zero space, which is equivalent to the topology being Hausdorff. 

Conversely, if $V$ is a vector space endowed with a topology as in the statement, the spaces $V_i$ can be taken quotients of $V$ by the elements of the topological subbasis of linear open neighborhoods of $0$, and the maps $v_i$ are the canonical projections. Being a quotient by an open subspace, the vector spaces $V_i$ are discrete. Moreover as we started from a subbasis for the topology, the topology on $V$ is exactly the coarsest one for which the projections $v_i$ are surjective.

Let us now verify the correspondence for morphisms. By the first part of the theorem, a linear map $f:V\to W$ is continuous if and only if all compositions $w_j\circ f$ are continuous. Since the projections $v_{f_*(j)}$ are continuous by construction (or by the above correspondence) and since the maps $f_j$ (being maps between discrete spaces) are of course continuous, we have that $w_j\circ f=f_j\circ v_{f_*(j)}$ is continuous as well.

We leave it to the reader to verify that these correspondences define an equivalence of categories.
\end{proof}

\begin{example}\label{ex:tvs}
\begin{enumerate}[(i)]
\item Any vector space becomes a topological vector space when endowed with the discrete topology, in other words, we consider for a vector space $V$ the object $(id_V:V\to V)$ in $\TVS_k$. 
\item Let \(V\) be a vector space. The dual space \(V^*\) is a topological vector space for the finite topology: a topological basis is given by the subspaces
\[V^*_F = \{\varphi \in V^* \mid \varphi(v) = 0 \ \text{ for all } v \in F\}\]
for every finite \(F \subseteq V\). This is the coarsest topology for which all $v:V^*\to k, \varphi\mapsto \varphi(v)$ are continuous for the maps $v\in V$ (where $k$ is considered with the discrete topology). Still otherwise stated, we have $(v:V^*\to k)_{v\in V}$ as object in $\TVS_k$.
\item Let \(V\) be a vector space and consider the \textit{weak} topology, i. e. a topological basis is given by the family of subspaces of finite codimension. This is the weakest topology that makes every element of \(V^*\) continuous (when \(k\) has the discrete topology). This is also the topology on \(V\) induced from the inclusion \(V \hookrightarrow V^{**},\) where \(V^{**}\) has the finite topology. As an object in $\TVS_k$, we have the object $(v^*:V\to k)_{v^*\in V^*}$.
\item Let $(V_i)_{i\in I}$ be a family of vector spaces and $V=\prod_i V_i$ their product. Then $(\pi_i:V\to V_i)_{i\in I}$ is a topological vector space where $\pi_i$ are the canonical projections. Remark that for any vector space \(W,\) \(W^* \cong \prod_{\dim W} k\) as topological vector spaces, where \(W^*\) has the finite topology. 
\item Let $(V_i)_{i\in I}$ be a family of vector spaces and $V=\bigoplus_i V_i$ their direct sum.
We can view \(V\) as a topological subspace of the topological vector space \(\prod_i V_i\) from the previous example (i.e.\ we consider the restrictions of the projections \(\pi_i\) to \(V\)).
Then, writing for a given vector space $W\cong \bigoplus_B k$ by fixing a base $B$, we recover in this way the weak topology on $W$. Remark that this construction is different from the coproduct in the category of topological vector spaces (as described in \cite{takeuchi}). Indeed, for this purpose, one should consider the topology on \(V=\bigoplus_i V_i\) given by \((\pi_J : V \to V_J)_{J \in P(I)}\) (where \(P(I)\) is the full powerset on \(I\)). In particular, the coproduct in $\TVS_k$ of a family discrete vector spaces is again discrete.
\end{enumerate}
\end{example}

The procedures in \cref{ex:tvs} are functorial. 
Moreover, the functor \(D : \Vect_k \to \TVS_k\) endowing a vector space with the discrete topology is a fully faithful left adjoint to the forgetful functor \(U:\TVS_k \to \Vect_k\). The functor \((-)^*:\Vect_k^{op} \to \TVS_k : V \mapsto V^*\) from example (ii)
is also fully faithful and is a right adjoint to the functor $(-)^{\odot}:\TVS_k\to \Vect_k^{op}$ that sends a topological vector space $V$ to the space of all continuous linear maps from $V$ into $k$. In particular, for any vector space $V$, we have a natural isomorphism $V\cong (V^*)^\odot$. These constructions show that we can view both $\Vect_k$ and $\Vect_k^{op}$ as a full subcategory of the category of topological vector spaces. This can be summarized in the following commutative diagram of adjoint functors (where the inner and outer triangle commutes).
\[
\xymatrix{
\Vect_k \ar@<-.8ex>[ddrr]_-D \ar@{}[rrrr]|-{\bot}
\ar@<.8ex>[rrrr]^{(-)^*} &&&& \Vect_k^{op} \ar@<-.8ex>[ddll]_-{(-)^*} \ar@<.8ex>[llll]^-{(-)^*} \ar@{}[ddll]|-{\sevdash} \\
\\
&& \TVS_k \ar@<-.8ex>[uull]_-U 
\ar@{}[uull]|-{\swvdash} \ar@<-.8ex>[rruu]_-{(-)^\odot}
}
\]

\begin{definition}\label{def:completion}
Let $\ul V=(v_i:V\to V_i)_{i\in V_*}$ be a topological vector space. Consider $P_f(V_*)$ with the directed partial order induced by containment. Then we define the {\em completion} of $V$ as the topological vector space $(\hat v_I:\hat V\to V_I)_{I\in P_f(V_*)}$, where the vector space $\hat V$ is given by the projective limit
\[\hat{V} = \varprojlim_{I\in P_f(V_*)} V_I.\]
and where we denote by $\hat v_I:\hat V\to V_I$ the canonical projections. 

By the universal property of the projective limit, we obtain a canonical map $\gamma_V: V\to \hat V$ such hat $v_I=\hat v_I\circ \gamma_V$ for all $I\in P_f(V_*)$, which means that $\gamma_V$ is a morphism in $\TVS_k$. We say that a topological vector space is {\em complete} if $\gamma_V$ is an isomorphism in $\TVS_k$. 

We can consider the category $\CTVS_k$ of complete topological vector spaces as a full subcategory of $\TVS_k$. 
\end{definition}

The following result follows directly from the definitions.

\begin{lemma}\label{le:completion}
For any topological vector space $\ul V$ the following statements hold.
\begin{enumerate}[(i)]
\item The completion \(\hat{V}\) is a complete topological vector space. 
\item The canonical map $\gamma_V$ is open and injective.
\item The canonical map is surjective if and only if \(V\) is complete.
\end{enumerate}
Moreover, the completion $\widehat{(-)}:\TVS_k\to \CTVS_k$ defines a left adjoint to the (fully faithful) forgetful functor $\CTVS_k\to \TVS_k$, and the unit of this adjunction is exactly given by the canonical map $\gamma$.
\end{lemma}

\begin{examples}
Let $V$ be a vector space. Then the discrete topological vector space $D(V)$ is complete, as is the topological vector space $V^*$ with the finite topology. The weak topology on $V$ is not complete unless $V$ is finite dimensional, and the completion is given by $V^{**}$ with its finite topology. The completion of a topological vector space $\bigoplus_{i\in I}V_i$ as in example \ref{ex:tvs}(v) (with the subspace topology) is given by the product space $\prod_{i\in I}V_i$.
\end{examples}

\begin{proposition}
The category $\TVS_k$ is a monoidal category such that the adjunction $U:\TVS_k\leftrightarrow \Vect_k:D$ is strict monoidal. Explicitly, the tensor product of two toplogical vector spaces $\ul V=(v_i:V\to V_i)_{i\in V_*}$ and $\ul W=(w_j:W\to W_j)_{j\in W_*}$ is given by 
$$\ul V\ot\ul W=((v\ot w)_{i,j}:=v_i\ot w_k:V\ot W\to V_i\ot W_j)_{(i,j)\in V_*\times W_*}$$
\end{proposition}

\begin{proof}
One can easily verify the statement holds. Let us just remark that since $\ker(v_i\ot w_k)=\ker v_i\ot W + V\ot \ker w_i$, the tensor product given here coincides with the one described in \cite[Section 1.5]{takeuchi}.
\end{proof}

Since the tensor product of two complete topological vector spaces is not necessarily complete, we also introduce a second tensor product between topological vector spaces.

\begin{definition}
Let $V$ and $W$ be two topological vector spaces. The {\em completed tensor product} is the completion of the tensor product:
$$V\hot W=\widehat{V\ot W}.$$
\end{definition}

The completed tensor product turns the category of complete topological vector spaces \(\CTVS_k\) into a monoidal category and the adjunction $\widehat{(-)}:\xymatrix@C=5pt{ (\TVS_k,\ot) \ar@<.8ex>[rr] & \scriptstyle{\bot} & (\CTVS_k,\hot) \ar@<.8ex>[ll] } : U$ is lax monoidal.
Remark however that \(\TVS_k\) is not a monoidal category with respect to $\hot$: since \(V \hotimes k = \hat{V},\) there can only be an isomorphism \(V \hotimes k \cong V\) if \(V\) is complete. 

\begin{lemma}[{\cite[Corollary 1.3, Theorem 2.3]{takeuchi}}]
	\label{le:projlim}
	\label{le:producttensor}
	\begin{enumerate}[(i)]
\item	Let \((V_i)_i\) and \((W_j)_j\) be families of topological vector spaces. Then 
	\[\left(\prod_i V_i\right) \hotimes \left(\prod_j W_j\right) \cong \prod_{i, j} V_i \hotimes W_j.\]
\item	Let \((V_i)_i\) be a projective system of complete topological vector spaces and let \(W\) be any topological vector space. Then there is a canonical isomorphism of topological vector spaces
	\[(\varprojlim V_i) \hotimes W \cong \varprojlim (V_i \hotimes W).\]
\end{enumerate}
In particular, the functor \(- \hotimes W\) preserves kernels of continuous linear maps between complete topological vector spaces, and intersections of complete topological vector spaces. 
\end{lemma}

In \cite{takeuchi}, Takeuchi introduced topological coalgebras. 
\begin{definition}
	A {\em topological coalgebra} is a triple \((C, \Delta, \epsilon)\), where \(C\) is a topological vector space, and 
	\begin{gather*}
		\Delta : C \to C \hotimes C \\
		\epsilon: C \to k
	\end{gather*}
are continuous linear maps that satisfy the usual coassociativity condition and counitality \((C \hotimes \epsilon)\Delta = (\epsilon \hotimes C) \Delta = \gamma_C\) as maps \(C \to \hat{C}\). We denote the category of topological coalgebras by $\TCoalg_k$
\end{definition}

\begin{example}
	\label{ex:dualcoalgebra}
Let \(A\) be a \(k\)-algebra. While the dual vector space \(A^*\) is not a coalgebra unless \(A\) is finite dimensional, \(A^*\) always has the structure of a topological coalgebra. Indeed, the finite topology on \(A^*\) coincides with the product topology when it is viewed as the product of \(\dim A\) copies of \(k\). By \cref{le:producttensor}, 
\[A^* \hotimes A^* \cong \prod_{(\dim A)^2} k \hotimes k \cong (A \otimes A)^*,\]
hence the multiplication map \(\mu : A \otimes A \to A\) dualizes to a map
\[\mu^* : A^* \to (A \otimes A)^* \cong A^* \hotimes A^*\]
which makes \(A^*\) into a topological coalgebra. 
\end{example}

Of particular interest to this article will be the topological cofree coalgebra. It turns out that its description is much simpler than that of the usual cofree coalgebra (see \cite[section 2.2]{coreps} and references therein). Let \(V\) be a topological vector space and write
\begin{align*}
	V^{\hotimes 0} &= k, \\
	V^{\hotimes 1} &= V, \\
	V^{\hotimes n} &= V \hotimes V^{\hotimes (n - 1)} \quad \text{ for } n > 1.
\end{align*}
Then consider the (complete) topological vector space
\[\hat{C}(V) = \prod_{n \in \N} V^{\hotimes n}\] 
which comes with the natural projection \(p : \hat{C}(V) \to V\). We then have the following result

\begin{theorem}\cite[Theorem 3.1]{takeuchi}
Let \(V\) be a topological vector space. The topological vector space $\hat C(V)$ constructed above can be endowed with the structure of a topological coalgebra such that for any topological coalgebra \(C\) and continuous linear map \(f : C \to V,\) there is a unique topological coalgebra morphism \(\bar{f} : C \to \hat{C}(V)\) such that \(f = p\bar{f}\). In other words, $\hat C:\TVS_k\to \TCoalg_k$ is a right adjoint to the forgetful functor $\TCoalg_k\to \TVS_k$.
\end{theorem}

Let us finish this section by a more explicit description of the topological cofree coalgebra over a discrete topological vector space.

\begin{example}
	\label{ex:cofree}
Let \(V\) be a vector space, and endow it with the discrete topology, that is we consider the object $\ul V=(id_V:V\to V)$ in $\TVS_k$. By definition we have that $\ul V\ot \ul V=(id_V\ot id_V:V\ot V\to V\ot V)$ is again discrete and hence complete which implies that $\ul V\ot \ul V=\ul V\hot \ul V$.
The cofree topological coalgebra on \(\ul V\) is then 
\[\hat{C}(V) = \prod_{n \geq 0} V^{\hat{\otimes} n} = \prod_{n \geq 0} V^{\otimes n}.\]
Denote by $\pi_i:\hat{C}(V) \to V^{\ot n}$ the natural projection on the $n$-th component of the product. Recall that $V^{\ot n}$ denotes an $n$-fold tensor product of copies of $V$. A generic element in this space is denoted by 
$$v^n_{(1)} \otimes \cdots \otimes v^n_{(n)}.$$ 
Remark that we have natural isomorphisms
$$\Delta_{i,j}:V^{\ot(i+j)}\to V^{\ot i}\ot V^{\ot j}$$
In order to distinguish the tensor product between $V^{\ot i}$ and $V^{\ot j}$ with the tensor product used to describe elements of $V^{\ot i}$ and $V^{\ot j}$, we will denote the former sometimes by $\boxtimes$. That is, we write sometimes $V^{\ot i}\ot V^{\ot j}=V^{\ot i}\boxtimes V^{\ot j}$ and we denote an element in this space as 
$$(v^{i}_{(1)} \otimes \cdots \otimes v^{i}_{(i)}) \boxtimes (v^{j}_{(1)} \otimes \cdots \otimes v^{j}_{(j)}).$$
Takeuchi's topological cofree coalgebra is a complete non-discrete topological coalgebra, where the comultiplication is a map
\[\Delta_{\hat{C}(V)} : \hat{C}(V) \to \hat{C}(V) \hotimes \hat{C}(V) = \prod_{j, k} V^{\otimes j} \otimes V^{\otimes k}\]
defined as follows: an element \((v^n_{(1)} \otimes \cdots \otimes v^n_{(n)})_n\) is sent to the element with \((j, k)\)-component
\begin{equation}
	\label{eq:Delta}
	(v^{j + k}_{(1)} \otimes \cdots \otimes v^{j + k}_{(j)}) \boxtimes (v^{j + k}_{(j + 1)} \otimes \cdots \otimes v^{j + k}_{(j + k)}).
\end{equation}
In other words, we have for any $i,j,n\in \NN$ with $n=i+j$ the following commutative diagram
\[
\begin{tikzcd}
	\hat C(V) \arrow{rr}{\Delta} \arrow{d}[swap]{\pi_n} && \hat C(V)\hat \ot\hat C(V) \arrow{d}{\pi_{i,j}:=\pi_i\hat\ot \pi_j} \\
	V^{\ot n} \arrow{rr}[swap]{\Delta_{i,j}} && 
	V^{\ot i}\ot V^{\ot j}
\end{tikzcd}
\]
The counit is given by projection on \(V^{\otimes 0} = k\), i.e. $\epsilon=\pi_0$, and there is the projection
\[p=\pi_1 : \hat{C}(V) \to V.\]
\end{example}

%

\subsection{Construction of the right adjoint}\label{se:constructionR}

In \cref{th:comonad}, we have shown that the forgetful functor 
\[U : \PMod^H \to \Vect_k\]
has a right adjoint
\[R : \Vect_k \to \PMod^H\]
and is comonadic. The aim of this section is to give an explicit description of the functor $R$ and the associated comonad \(\C = UR\). 
We start with a technical lemma. 
\begin{lemma}
	\label{le:biggestsub}
	Let $M$ be a $k$-vector space endowed with a linear map $\rho:M\to M\ot H$ (no axioms required).
	The largest subspace $N$ of $M$ (with respect to inclusion) such that $\rho(N)\subseteq N\ot H$ and $(N,\rho)$ is a partial $H$-comodule, is the set of all elements $m \in M$ that satisfy, for all $n\geq 1$ the following identities:
	\begin{enumerate}[({A}1)]
		\item $(M \otimes \epsilon \otimes H^{\ot(n - 1)})\rho^n(m) =\rho^{n-1}(m)$; \label{A1}
		\item $(M \otimes H \otimes \mu \otimes H^{\ot(n - 1)})(M \otimes H \otimes H \otimes S \otimes H^{\ot(n-1)})(M \otimes \Delta \otimes H^{\ot n}) \rho^{n + 1}(m)  \\ =(M \otimes H \otimes\mu \otimes H^{\ot (n - 1)})(M \otimes H \otimes H \otimes S \otimes H^{\ot(n - 1)})\rho^{n+2}(m)$; \label{A2}		
		\item $(M \otimes \mu \otimes H^{\ot n})(M \otimes H \otimes S \otimes H^{\ot n})(M \otimes H \otimes \Delta \otimes H^{\ot(n - 1)}) \rho^{n+1}(m)\\ =(M \otimes \mu \otimes H^{\ot n})(M \otimes H \otimes S \otimes H^{\ot n})\rho^{n+2}(m)$. \label{A3}	
	\end{enumerate}
\end{lemma}
\begin{proof}
	Suppose that \(N \subseteq M\) is a subspace satisfying axioms \ref{A1}, \ref{A2} and \ref{A3} for all \(n \geq 1\). Then, since the functor \(- \otimes H\) preserves kernels of linear maps, \(N \otimes H\) is the set of elements \(x\in M \otimes H\) for which the following identities hold for all \(n \geq 1\) 
	\begin{gather}
		([(M \otimes \epsilon \otimes H^{n - 1}) \rho^n] \otimes H)(x) = (\rho^{n - 1} \otimes H)(x); \label{A1H} \\
		([(M \otimes H \otimes \mu \otimes H^{n - 1})(M \otimes H \otimes H \otimes S \otimes H^{n-1})(M \otimes \Delta \otimes H^{n}) \rho^{n + 1}] \otimes H)(x) \nonumber \\ =([(M \otimes H \otimes\mu \otimes H^{n - 1})(M \otimes H \otimes H \otimes S \otimes H^{n - 1})\rho^{n+2}] \otimes H)(x); \label{A2H}\\
		([(M \otimes \mu \otimes H^{n})(M \otimes H \otimes S \otimes H^{n})(M \otimes H \otimes \Delta \otimes H^{n - 1}) \rho^{n+1}] \otimes H)(x) \nonumber\\ = ([(M \otimes \mu \otimes H^{n})(M \otimes H \otimes S \otimes H^{n})\rho^{n+2}] \otimes H)(x). \label{A3H}
	\end{gather}
	Take now any $m\in N$. The above equations for $x=\rho(m)$ become exactly the axioms \ref{A1}, \ref{A2} and \ref{A3} with \(n\) replaced by \(n + 1\). As these are all satisfied, we find that \(\rho(N) \subseteq N \otimes H\). Now \((N, \rho)\) is a partial comodule because \ref{A1}, \ref{A2} and \ref{A3} for \(n = 1\) are the axioms \ref{PCM1}, \ref{PCM2} and \ref{PCM3} of partial comodules.
	
	Conversely, suppose that \(N\subseteq M\) such that $(N,\rho|_N)$ is a partial comodule. We will show by induction on \(n\) that the elements of $N$ satisfy \ref{A1}, \ref{A2} and \ref{A3}. By definition, the axioms hold for \(n = 1\). When \ref{A1} is satisfied for \(n\), then it follows from (\ref{A1H}) that
	\begin{align*}
		([(M \otimes \epsilon \otimes H^{n - 1}) \rho^n] \otimes H) \rho &= (\rho^{n - 1} \otimes H) \rho \\
		\Rightarrow \quad (M \otimes \epsilon \otimes H^n)\rho^{n + 1} &= \rho^n.
	\end{align*}
	Similarly, it follows from (\ref{A2H}) and (\ref{A3H}) that \ref{A2} and \ref{A3} hold for \(n + 1\).
	
	To see that the construction of this lemma coincides with the ``largest partial comodule inside $M$", it suffices now to observe that if $N$ and $N'$ are two partial comodules with $N,N'\subseteq M$ and such that their coactions coincide with the restriction of $\rho$, then $N+N'$ is also such a partial comodule. So we can take the sum of all these and get the largest partial comodule inside $M$. By the above, this sum consists exactly of all elements in $m$ satisfying \ref{A1}, \ref{A2} and \ref{A3}.
\end{proof}

\begin{remark}
    \label{rm:generalA1}
    Applying \(\rho^\ell \otimes H^{\otimes(n - 1)}\) to both sides of \ref{A1} shows that this condition implies
    \[(M \otimes H^{\otimes \ell} \otimes \epsilon \otimes H^{\ot(n - 1)})\rho^{n + \ell}(m) =\rho^{n + \ell -1}(m)\]
    and a similar trick can be applied to \ref{A2} and \ref{A3}. 
\end{remark}

To construct a right adjoint to the forgetful functor \(U : \PMod^H \to \Vect_k,\) we should consider for every vector space \(V\), a linear map \(\rho : M \to M \otimes H,\) where \(M\) is a `large' space associated to \(V\). Applying then \cref{le:biggestsub} leads to a partial comodule. A first idea could be to consider \(V \otimes C(H),\) where \(C(H)\) is the cofree coalgebra over the vector space \(H\), and the linear map
\[\begin{tikzcd}
	V \otimes C(H) \arrow{rr}{V \otimes \Delta_{C(H)}} && V \otimes C(H) \otimes C(H) \arrow{rr}{V \otimes C(H) \otimes p} && V \otimes C(H) \otimes H.
\end{tikzcd}\]
However, this construction does not give an adjoint to \(U:\) when \(M\) is an irregular partial comodule, there is no way to define the unit \(\eta_M : M \to M \otimes C(H)\). 




To circumvent this problem, we will use the topological cofree coalgebra \(\hat{C}(H)\) instead, where we endow the vector space \(H\) with the discrete topology. 
%
%
%
Remark that in this case \(\Delta_{\hat{C}(H)}\) is completely unrelated to the comultiplication of the Hopf algebra \(H\).

Consider the object
\[V \hotimes \hat{C}(H) = \prod_{n \geq 0} V \otimes H^{\otimes n} = \{f\in \Hom_k(k[t]\to \bigoplus_{n\in \NN} V\ot H^{\ot n})~|~f \text{ is of degree 0}\}.\]
Remark that degree $0$ means that $f(t^{m})\in V\ot H^{\ot m}$ for all $m\in \NN$. In this paper, we will use the product notation and avoid to work with degree zero functions. The \(k\)-component of an element \(f \in \prod_i V_i\) will be denoted as \(f^k \in V_k\).

Define the map \(\sigma\) as the composition
\begin{equation}
	\label{eq:sigma}
\begin{tikzcd}
	V \hotimes \hat{C}(H) \arrow{rr}{V \hotimes \Delta_{\hat{C}(H)}} && V \hotimes \hat{C}(H) \hotimes \hat{C}(H) \arrow{rr}{V \hotimes \hat{C}(H) \hotimes p} && V \hotimes \hat{C}(H) \hotimes H.
\end{tikzcd}
\end{equation}

We define \(R^0(V)\) as the largest linear subspace of \(V \hotimes \hat{C}(H)\) such that \(\sigma(R^0(V)) \subseteq R^0(V) \otimes H\). In other words, \(R^0(V)\) is the sum of all subspaces $M$ of \(V \hotimes \hat{C}(H)\) such that $\sigma(M)\subseteq M\ot H$.


\begin{lemma}
	\label{le:R0}
	The space \(R^0(V)\) is the set of elements \(f \in V \hotimes \hat{C}(H)\) such that for every \(m \geq 0\), there are finite families of elements \(g_i \in V \hotimes \hat{C}(H), h_i \in H^{\otimes m}\) such that for all \(k \geq m\)
	\begin{equation}
		\label{eq:representative}
		f^k = \sum_i g_i^{k - m} \otimes h_i.
	\end{equation}
\end{lemma}
\begin{proof}
	First, let us remark that $\sigma$ sends an element \(f \in V \hotimes \hat{C}(H)\) with components (\(n \geq 0\))
	\[f^n = \pi_n(f) =  v^n_{(0)} \otimes v^n_{(1)} \otimes \cdots \otimes v^n_{(n)}\] 
	to the element of \(V \hotimes \hat{C}(H) \hotimes H = \prod_{m \geq 0} (V \boxtimes H^{\otimes m} \boxtimes H)\) with components (\(m + 1 = n \geq 1\))
	\[\sigma(f)^n=\pi_n(\sigma(f))=v^n_{(0)} \boxtimes \left(v^n_{(1)} \otimes \cdots \otimes v^n_{(n - 1)}\right) \boxtimes v^n_{(n)}.\] 
	Hence for this element to be in \((V \hotimes \hat{C}(H)) \otimes H\), there need to exist finite families of elements \(g_i \in V \hotimes \hat{C}(H)\) and \(h_i \in H\) such that $\sigma(f)=\sum_i g_i\ot h_i$. In view of the explicit form of $\sigma$ presented above, this means more precisely that for all \(k \geq 1\) we have an identity
	\[f^k = \sum_i g_i^{k - 1} \otimes h_i.\]
	We call the set of all elements $f\in V\hotimes \hat C(H)$ with this property $R_1(V)$. More generally, for each $j\ge 1$, we denote by $R_j(V)$ the set of all \(f \in V \hotimes \hat{C}(H)\) such that for every \(m \leq j\), there are finite families of elements \(g_i \in V \hotimes \hat{C}(H), h_i \in H^{\otimes m}\) such that for all \(k \geq m\)
	\[f^k = \sum_i g_i^{k - m} \otimes h_i.\]
	From the definition it is then clear the we have a descending chain of subspaces
	$$R_1(V)\supseteq R_2(V)\supseteq \cdots \supseteq R_j(V)\supseteq R_{j+1}(V) \supseteq \cdots$$
	Our aim is to prove that 
	\[R^0(V) = \bigcap_{j \geq 1} R_j(V).\]
	Let us start by showing by induction that $R^0(V)$ is contained in the intersection. From the observations above we know that $R^0(V)\subseteq R_1(V)$. 
	 Suppose now that \(R^0(V) \subseteq R_j(V)\) for some \(j \geq 1\) and let us show that \(R^0(V) \subseteq R_{j + 1}(V)\). Indeed, by the definition of $R^0(V)$ and our new assumption, the image of an element \(f \in R^0(V)\) under \(\sigma\) should end up in \(R^0(V)\ot H\subseteq R_j(V) \otimes H\), hence there exist finite families of elements \(g_i \in R_j(V), h_i \in H\) such that for \(k \geq 1\)
	\[f^k = \sum_i g_i^{k - 1} \otimes h_i.\]
	Now fix any $m\le j$. As $g_i \in R_j(V)$, we know that there exist some \(g'_{\ell_i} \in V \hotimes \hat{C}(H)\), \(h'_{\ell_i} \in H^{\otimes m}\) such that for all \(k \geq m + 1\) 
	we have \(g_i^{k - 1} = \sum_{\ell_i} {g'}_{\ell_i}^{k - m - 1} \otimes h'_{\ell_i}\). Combining both identities, we find that 
	$$f^k =\sum_{i,\ell_i} {g'}_{\ell_i}^{k - m - 1} \otimes h'_{\ell_i} \ot h_i \in V\hot \hat C(H) \ot H^{\ot m+1}.$$
	As these elements $g'_{\ell_i}$ and $h'_{\ell_i} \ot h_i\in H^{\ot m+1}$ can be found for any $m\le j$ (i.e. any $m+1\le j+1$), 
	we may conclude that \(f \in R_{j + 1}(V)\).
	This shows that
	\[R^0(V) \subseteq \bigcap_{j \geq 1} R_j(V).\]
	Conversely, take any \(f \in \bigcap_{j \geq 1} R_j(V)\). We want to show that \(\sigma(f) \in \bigcap_{j \geq 1} R_j(V) \otimes H\), which will conclude the proof by the maximality of $R^0(V)$. 
	Since $f\in R_1(V)$, we can write
	\[f^k = \sum_i g_i^{k - 1} \otimes h_i\]
	where the \(h_i \in H\) are supposed to be linearly independent. 
	Take \(m \geq 1\). Since also $f\in R_{m}(V)$, there are \(g'_j \in V \hotimes \hat{C}(H)\), \(h'_j \in H^{\otimes m}\) such that for all $k\ge m$
	\[f^k = \sum_j {g'}_j^{k - m} \otimes h'_j.\]
	We can write \(h'_j = \sum_i h''_{ij} \otimes h_i\) for some \(h''_{ij}\in H^{\otimes (m - 1)}\), and by the linear independence of the elements $h_i$ it follows that that
	\[g_i^{k - 1} = \sum_j {g'}_j^{k - m} \otimes h''_{ij},\]
	which shows that \(g_i \in \bigcap_{j \geq 1} R_j(V)\).	
\end{proof}

\begin{remarks}
	\label{rm:RV}
\begin{enumerate}[(i)]
    \item The description of \(R^0(V)\) in \cref{le:R0} should be compared to the usual cofree comodule \(V \otimes C(H)\), where \(C(H)\) is the usual cofree coalgebra over \(H\). This space can be described as the subset of elements \(f \in V \hotimes \hat{C}(H)\) such that there exist finite families of elements \(g_i \in V \otimes C(H)\) and \(h_i \in C(H)\) such that for every \(m \geq 0, k \geq m\)
	\[f^k = \sum_i g_i^{k - m} \otimes h_i^m.\]
	Therefore, $f$ satisfies (\ref{eq:representative}), so \(V \otimes C(H) \subseteq R^0(V)\).
	\item When \(H\) is finite dimensional, then the representability condition (\ref{eq:representative}) becomes trivial and \(R^0(V) = V \hotimes \hat{C}(H)\). Indeed, let \(\{h_1, \dots, h_n\}\) be a basis for \(H\) and fix any $m\ge 0$. Consider all tuples $\ul i:=(i_1,\ldots,i_m)$ where  \(1 \leq i_1, \dots, i_m \leq n\) and denote \(h_{\ul i}=h_{i_1} \otimes \cdots \otimes h_{i_m} \in H^{\otimes m}\). Then we define $g_{\ul i}\in V\hotimes \hat C(H)$ by putting for each degree $k\ge 0$,
	\[ g_{\ul i}^k := (V\ot H^{\ot k}\ot h^*_{i_1}\ot \cdots h^*_{i_m})(f^{k+m})\]
    where the $h^*_i$ are the dual base elements for $H$. Then we find indeed that  for all $k\ge m$, $f^k =\sum_{\ul i} g_{\ul i}^{k-m}\ot h_{\ul i}$.
\end{enumerate}
	
\end{remarks}

We now define \(R(V)\) to be the largest subspace of \(R^0(V)\) such that \((R(V), \sigma)\) becomes a partial comodule (see Lemma \ref{le:biggestsub} for a characterization of this subspace).

\begin{lemma}\label{functorR}
	With notation as above, \(R\) defines a functor \(\mathsf{Vect}_k \to \mathsf{PMod}^H\), on objects given by \(V \mapsto R(V)\), and on morphisms by \(\varphi \mapsto (\varphi \hotimes \hat{C}(H))_{|R(V)}\).
\end{lemma}

\begin{proof}
    By construction, $R(V)$ is a partial $H$-comodule for any vector space $V$. Given a linear map \(\varphi : V \to W\), consider
	\[\varphi \hotimes \hat{C}(H) : V \hotimes \hat{C}(H) \to W \hotimes \hat{C}(H),\]
	which clearly commutes with \(\sigma\); more precisely
	\begin{equation}\label{commuteswithsigma}
	(\varphi \hotimes \hat{C}(H) \hotimes \hat{C}(H) ) \sigma_V = \sigma_W (\varphi \hotimes \hat{C}(H))
	\end{equation}
	We have to show that the image of the restriction of $\varphi \hotimes \hat{C}(H)$ to $R(V)$ lies in $R(W)$.
	Take \(f \in V \hotimes \hat{C}(H)\) with components
	\[f^n = v^n_{(0)} \otimes v^n_{(1)} \otimes \cdots \otimes v^n_{(n)}.\]
	Then \((\varphi \hotimes \hat{C}(H))(f)\) has components
	\[\varphi(v^n_{(0)}) \otimes v^n_{(1)} \otimes \cdots \otimes v^n_{(n)}.\]
	When \(f \in R^0(V)\), then clearly \((\varphi \hotimes \hat{C}(H))(f) \in R^0(W)\), where the needed finite family \(g'_i \in W \hotimes \hat{C}(H)\) is found by applying \(\varphi \hotimes \hat{C}(H)\) to \(g_i \in V \hotimes \hat{C}(H)\). Moreover, by \eqref{commuteswithsigma}, if  $f$ satisfies the conditions of \cref{le:biggestsub} (i.e.\ \(f \in R(V)\)), then so does  \((\varphi \hotimes \hat{C}(H))(f)\). This shows that $R(\varphi):= (\varphi \hotimes \hat{C}(H))_{|R(V)}$ is a morphism in $\PMod^H$. The functoriality is clear. 
\end{proof}

\begin{theorem}
	\label{th:adjoint}
	The functor \(R\) from \cref{functorR} is a right adjoint to the forgetful functor $U:\PMod^H\to \Vect_k$.
\end{theorem}

\begin{proof}	
	We start by defining the unit for the adjunction.
	For a partial comodule \((M, \rho)\), consider the linear map
	\[\eta_M : M \to M \hotimes \hat{C}(H) : m \mapsto (\rho^n(m))_{n \geq 0}.\]
	We claim that \(\eta_M\) corestricts to a partial comodule morphism \(M \to RU(M)\).
	
	Let us first show that \(\eta_M(m) \in R^0U(M)\). Fix a basis \(\{h_i \mid i \in I\}\) for \(H\). Then
	\[\rho(m) = \sum_{j_1 \in J_1} m_{j_1} \otimes h_{j_1}\]
	for a finite \(J_1 \subseteq I\) and some elements \(m_{j_1}\in M\). Inductively we get for every \(k \geq \ell\)
	\[\rho^{k}(m) = (\rho^{k - \ell} \otimes H^{\otimes \ell}) \rho^{\ell}(m) = \sum_{j_1 \in J_1, \dots, j_{\ell} \in J_\ell} \rho^{k - \ell}(m_{j_1 \cdots j_\ell}) \otimes h_{j_\ell} \otimes \cdots \otimes h_{j_1}\]
	for finite sets \(J_1, \dots, J_\ell \subseteq I\) and \(m_{j_1 \cdots j_\ell} \in M\).
%
	Now fix any $\ell\ge 0$ and take
	\begin{align*}
		g_{j_1 \cdots j_\ell} = (\rho^{k - \ell}(m_{j_1 \cdots j_\ell}))_{k \geq \ell} &\in UM\hotimes \hat C(H) \\
		h_{j_1 \cdots j_\ell} = h_{j_1} \otimes \cdots \otimes h_{j_\ell} &\in H^{\ot \ell}
	\end{align*}
	then it is clear that with these elements \(\eta_M(m)\) satisfies \eqref{eq:representative}, hence \(\eta_M(m) \in R^0U(M)\) by \cref{le:R0}.
	
	To show that \(\eta_M(m)\) satisfies the list of axioms of Lemma \ref{le:biggestsub}, remark that for \(f = (\rho^k(m))_k\), we have
	\begin{equation}
		\label{eq:actionsigma}
		\sigma^\ell(f) = \sum_{j_1 \in J_1, \dots, j_{\ell} \in J_\ell} (\rho^{k - \ell}(m_{j_1 \cdots j_\ell}))_{k \geq \ell} \otimes h_{j_\ell} \otimes \cdots \otimes h_{j_1},
	\end{equation}
which is in essence a rewriting of \((\rho^k(m))_{k}\).
	Because \((M, \rho)\) is a partial comodule, it satisfies the identities of \cref{le:biggestsub} and those from \cref{rm:generalA1}. 
	Now it follows from (\ref{eq:actionsigma}) that \(f\) satisfies the identities of \cref{le:biggestsub} with respect to the coaction \(\sigma\).
	
	Let us show that \(\eta_M\) is a morphism of partial comodules. We need to show that the diagram
	\[
	\begin{tikzcd}
		M \arrow{r}{\eta_M} \arrow{d}{\rho} & RU(M) \arrow{d}{\sigma} \\
		M \otimes H \arrow{r}{\eta_M \otimes H} & RU(M) \otimes H
	\end{tikzcd}
	\]
	commutes. On one hand we find \((\eta_M \otimes H) \rho(m) = \sum_{j \in J_1} (\rho^n(m_j))_n \otimes h_j\), and on the other hand it follows from \eqref{eq:actionsigma} that \[\sigma  \eta_M(m) = \sigma((\rho^n(m))_n) = \sum_{j \in J_1} (\rho^n(m_j))_n \otimes h_j.\] Finally, when \(f : M \to N\) is a morphism between the partial comodules \((M, \rho_M)\) and \((N, \rho_N),\) then \(\rho^n_N f = (f \otimes H^n)\rho_M^n\) for every \(n \in \N,\) hence \(\eta_N f = RU(f) \eta_M,\) which shows that \(\eta\) is a natural transformation. 
	
	The counit of the adjunction will be the restriction of 
	\[V \hotimes \epsilon_{\hat{C}(H)} : V \hotimes \hat{C}(H) \to V \hotimes k = V\]
	to a linear map \(\epsilon_V : UR(V) \to V\), which is the restriction of the projection $\pi_0$. 
	
	To finish the proof, we have to check the unit-counit relations.
	Let \(M\) be a partial comodule. Then the composition of linear maps
	\[
	\begin{tikzcd}
		U(M) \arrow{rr}{U(\eta_M)} && URU(M) \arrow{rr}{\epsilon_{U(M)}} && U(M)
	\end{tikzcd}\]
	maps an \(m \in M\) to \(\epsilon_{U(M)}((\rho^n(m))_n) = \rho^0(m) = m\).
	
	For the second unit-counit condition, let \(V\) be a vector space and recall that 
	$R(V)\subseteq V\hot \hat C(H)=\prod_n V\ot H^{\ot n}$. For a topological vector space \(W\), we denote by \(W_d\) the underlying vector space endowed with the discrete topology. Then $RUR(V)\subseteq (V\hot \hat C(H))_d \hot \hat C(H)=\prod_{\ell} \left(\prod_k V\ot H^{\ot k}\right) \ot H^{\ot \ell}$.
	Let \(\{h_i \mid i \in I\}\) a basis for \(H\). For an element \(f \in R(V)\) with components
	\[f^n = \sum_{j_1 \in J_1, \dots, j_n \in J_n} v^n_{j_n \cdots j_1} \otimes h_{j_n} \otimes \cdots \otimes h_{j_1}\]
	where each \(J_i \subseteq I\) is finite (these sets exist exactly because \(f \in R^0(V)\)),
	we find that $\eta_{R(V)}(f)\in RUR(V)$ is given by the element of $(V\hot \hat C(H))_d \hot \hat C(H)=\prod_{\ell} \left(\prod_k V\ot H^{\ot k}\right) \ot H^{\ot \ell}$ with the following \(\ell\)-components:
	\[\sum_{j_1 \in J_1, \dots, j_{k + \ell} \in J_{k + \ell}} \left(v^{k + \ell}_{j_{k + \ell} \cdots j_{1}} \otimes h_{j_{k + \ell}} \otimes \cdots \otimes h_{j_{\ell + 1}}\right)_{k \geq 0} \otimes \left(h_{j_{\ell}} \otimes \cdots \otimes h_{j_{1}}\right).\]
	Then $R(\epsilon_V)$ projects the first tensorand on its $k=0$ component, so that we obtain
	$$[R(\epsilon_V) \eta_{RV}(f)]^\ell = \sum_{j_1 \in J_1, \dots, j_\ell \in J_\ell} v^\ell_{j_\ell \cdots j_1} \otimes h_{j_\ell} \otimes \cdots \otimes h_{j_1} =f^\ell$$
	as needed.
\end{proof}

As we have observed in \cref{th:comonad}, the functor $U:\PMod^H\to \Vect_k$ is comonadic, which means that the category of partial comodules is equivalent to the Eilenberg-Moore category over a comonad $\C$ on $\Vect_k$. Let us make this construction more concrete by means of the explicit form of the right adjoint of $U$ as described above. The functor underlying the comonad $\C$ is given by the composition $UR$. For any vector space $V$ we then find that $\C(V)$ is given by the subspace of $V\hot \hat C(H)=\prod_{n\ge 0} V\ot H^{\ot n}$ given by all elements satisfying the conditions of \cref{le:biggestsub} and \cref{le:R0}. The counit $\epsilon:\C=UR\to id$ of the comonad is given by the counit of the adjunction $(U,R)$, that is the projection on the $0$-th component. The comultiplication $\nabla_V:\C(V)\to\C\C(V)$ is given by 
$U\eta R:\C=UR\to URUR=\C\C$, where $\eta$ is the unit of the adjunction $(U,R)$. Take any $f\in \C(V)$. As $f\in R^0(V)$, for each $m\ge 0$, there exist $g^{(m)}_i\in V\hot \hat C(H)$ and $h_i^{(m)}\in H^{\ot n}$ such that for all \(n \geq m\)
$$f^n=\sum_i (g_i^{(m)})^{n - m}\ot h_i^{(m)}$$
Then the comultiplication is given by 
$$\nabla_V((f^n)_n) = \left(\sum_i (g^{(m)}_i)^n \ot h^{(m)}_i\right)_{n,m}\in \prod_{n,m\ge 0} V\ot H^{\ot n}\ot H^{\ot m}.$$


Recall that an object in the Eilenberg-Moore category \(\Vect_k^\C,\) is a pair \((V, \delta),\) where \(V\) is a vector space and \(\delta\) a linear map \(V \to \C(V)\) such that
\begin{gather}
	\epsilon_V \delta = V, \label{eq:EM1}\\
	\C(\delta) \delta = \nabla_V \delta. \label{eq:EM2}
\end{gather}
A morphism \((V, \rho) \to (V', \rho')\) in \(\Vect_k^\C\) is a linear map \(f : V \to V'\) such that \(\C(f) \rho = \rho' f\). 

Given a partial comodule \((M, \rho)\), the underlying vector space $M$ can be turned into an Eilenberg-Moore object over \(\C\) by means of the coaction
$$\delta=\eta_M: M\to \C(M)=UR(M), m\mapsto (\rho^n(m))_{n\ge 0},$$
which can be checked to satisfy (\ref{eq:EM1}) and (\ref{eq:EM2}). This construction yield a functor 
\(K : \PMod^H \to \Vect_k^\C\) which makes the following diagram commute (where $U$ denote the forgetful functors)
\[
\xymatrix{
\Vect_k^\C \ar[drr]_U && \PMod^H \ar[ll]_-K \ar@<.5ex>[d]^-U \\
&& \Vect_k \ar@<.5ex>[u]^-R
}
\]
By \cref{th:comonad}, the functor $K$ is an equivalence, even an isomorphism, of categories.
	The inverse of the comparison functor $K$ is described as follows: given an Eilenberg-Moore object \((V, \delta)\) in \(\Vect_k^\C\) (hence \(\delta\) is a linear map \(V \to \C(V) \subseteq V \hotimes \hat{C}(H)\)), 
	\[\rho = \pi_1 \delta : V \to V \otimes H\]
	defines a partial comodule structure on \(V\). (Recall that \(\pi_1 : \prod_{n \geq 0} V \otimes H^{\otimes n} \to V \otimes H\) is the projection on the \(1\)-component.)  


\subsection{Finite dimensional case}

\begin{proposition}\label{pr:finitedimensional}
	Let \(H\) be a finite dimensional Hopf algebra and denote as before the comonad associated to the adjunction between $\PMod^H$ and $\Vect_k$ by $\C$. Then we have an isomorphism of comonads \(\C \cong \Hom_k(H^*_{par}, -)\).
\end{proposition}
\begin{proof}
As we have seen (see \cref{th:comonad}), the category of partial $H$-comodules is isomorphic to the category of comodules over a comonad $\C$ on $\Vect_k$.
	Let \(H\) be a finite-dimensional Hopf algebra.  We know from \cite[Theorem 4.14]{coreps} that in this case, the category of partial $H$-comodules is isomorphic to the category of partial $H^*$-modules, which is in turn isomorphic to the category of $(H^*)_{par}$-modules. The latter is nothing else than the Eilenberg-Moore category over the monad $-\ot (H^*)_{par}$, which is known (see, for example \cite[Definition 2.22]{Gabibook}) to be isomorphic to the Eilenberg-Moore category over the comonad $\Hom_k((H^*)_{par}, -)$. Combining all these observations, we obtain the following commutative diagram of functors, where the horizontal functors are isomorphisms of categories and the vertical arrows (indicated by $U$) are forgetful functors.
	\[
	\begin{tikzcd}
	\Vect_k^\CC \arrow[leftrightarrow]{rr}{\cong} \arrow{drr}[swap]{U} && \PMod^H \arrow[leftrightarrow]{rr}{\cong} \arrow{d}{U} && \Vect_k^{\Hom_k((H^*)_{par},-)} \arrow{dll}{U}\\
	&& \Vect_k
	\end{tikzcd}
	\]
By uniqueness of adjoints, it follows that the isomorphisms of categories also commute with the `free' functors. From this, we can conclude that there is an isomorphism of comonads $\C\cong \Hom_k(H^*_{par}, -)$.

For sake of completion, let us state the explicit form of this isomorphism. Fixing a base \(\{e_1, \dots, e_n\}\) and dual base \(\{f_1, \dots, f_n\}\) for $H$, we can define for any vector space $V$ a map
	$$\alpha:\Hom_k(H^*_{par},V)\to V\hot \hat C(H)$$
	by 
	\[\alpha(\varphi)= \left( \sum_{i_1, \dots, i_k = 1}^n \varphi([f_{i_1}] \cdots [f_{i_k}]) \otimes e_{i_1} \otimes \cdots \otimes e_{i_k}\right)_{k\ge 0},\]
	One can check that this map corestricts to an isomorphism onto $R(V)$, whose inverse is given by 
	\[\beta:R(V)\to \Hom_k(H^*_{par}, V),\ f\mapsto ([f_{i_1}] \cdots [f_{i_k}] \mapsto (V \otimes f_{i_1} \otimes \cdots \otimes f_{i_k})(f^k)),\]
	which provides the required isomorphism of comonads.
	\end{proof}

\begin{remark}
	Endowing \(\Hom_k(H^*_{par}, V)\) with the product topology (where both \(H^*_{par}\) and \(V\) are discrete) and \(\C(V)\) with the induced topology from \(V \hotimes \hat{C}(H)\), then the isomorphism
	\[\C(V) \cong \Hom_k(H^*_{par}, V)\]
	is in fact an isomorphism of topological vector spaces. 
	
	Indeed, the sets (for all \(k \in \N\))
	\[U_k = \{f \in \C(V) \mid f^k = 0\}\]
	form a subbasis of neighbourhoods for \(\C(V)\) (i.~e.~the finite intersections of these sets give a basis of neighbourhoods) and they correspond to 
	\[\tilde{U}_k = \{\varphi : H^*_{par} \to V \mid \varphi([f_{i_1}] \cdots [f_{i_k}]) = 0 \ \forall i_1, \dots, i_k  \in \{1, \dots, n\}\}\]
	under the isomorphism. Since these sets form a subbasis of neighbourhoods for \(\Hom_k(H^*_{par}, V)\), the linear isomorphism is a homeomorphism. 
	
	This shows in particular that \(\C(V)\) is complete for any vector space \(V\), since it is linearly homeomorphic to a product of discrete spaces. In fact, when \(H\) is finite dimensional, \(R^0(V) = V \hotimes \hat{C}(H)\) (see \cref{rm:RV}), of which \(R(V)\) is a closed subspace as an intersection of kernels of continuous linear maps.
\end{remark}

\section*{Conclusions and future directions}\label{future}

As mentioned in the introduction (see \cite{ABV}), to any Hopf algebra $H$, one can associate an algebra $H_{par}$ whose category of modules is isomorphic to the category of partial modules over $H$. In \cite{coreps} it was proven that the category of all regular comodules over a Hopf algebra $H$ is isomorphic to the category of comodules over a coalgebra $H^{par}$ (see also \cref{regularH}). In particular, for a regular Hopf algebra the full category of partial comodules is isomorphic to the category of comodules over $H^{par}$.

In this paper we showed that for any Hopf algebra $H$, one can find a comonad $\CC$ whose Eilenberg-Moore category is equivalent to the category of partial $H$-modules. In the non-regular case however (such as Sweedler's Hopf algebra \cref{ex:irregular} and \(U_q(\mathfrak{sl}_2)^\circ\) \cref{ex:quantum}), this comonad is {\em not} induced by a coalgebra by \cref{regularH}. 

In case of a finite dimensional Hopf algebra $H$, we have shown in \cref{pr:finitedimensional} that the underlying functor of the comonad $\C$ is representable, that is, it is induced by an algebra $R$ such that $\C$ is naturally isomorphic to $\Hom_k (R,-)$. Therefore, the Eilenberg-Moore category of $\CC$ coincides in this case with the category of modules over $R$. Let us remark that in general (when the Hopf algebra $H$ is allowed to be infinite dimensional), no such algebra $R$ exists.
Indeed, a counter example of this situation is given by a non-trivial connected algebraic group $G$. In this case we know by Theorem \ref{th:connectedgroup} that the category of partial comodules over $\Oo(G)$ is just the category of (usual, global) comodules over $\Oo(G)$. If it would be equivalent to the Eilenberg-Moore category over $\Hom(R,-)$ for some algebra $R$, hence to the category of modules over $R$, then by \cref{equivcatfindim} we know that $R$, and and consequently $\Oo(G)$, need to be finite dimensional, which is only possible if $G$ is the trivial group.

These observations show that in order to describe partial comodules over arbitrary Hopf algebras, the comonadic approach developed in this paper is unavoidable. 

In \cite{ABV} it was moreover shown that $H_{par}$ can be endowed with the structure of a Hopf algebroid, and dually in \cite{coreps} it was shown that $H^{par}$ has the structure of a Hopf coalgebroid. These results imply in particular that the categories of partial modules and regular partial comodules over $H$ are closed monoidal.

In view of these results, it is tempting to expect that the comonad $\C$ constructed in this paper would have the structure of a Hopf comonad, which would turn the category of partial comodules into a closed monoidal category. The obstacle to obtain such a result is however that at this moment, it is unclear what would be the monoidal base category over which this Hopf comonad lives, and in particular what would be the monoidal unit. Indeed, the comonad we constructed is a comonad over $\Vect_k$. In the regular case, this comonad corresponds to the comonad $-\ot H^{par}$ on $\Vect_k$ induced by the coalgebra $H^{par}$ (similarly, we have the monad $-\ot H_{par}$ on $\Vect$ associated induced by the total algebra of the Hopf algebroid $H_{par}$). However, to invoke the the Hopf coalgebroid structure of $H^{par}$, one needs to lift this comonad to a comonad on the monoidal category of $C^{par}$-bicomodules, where $C^{par}$ is the `universal copar coalgebra' which was also constructed in \cite{coreps} and which is a suitable quotient of $H^{par}$. In our current setting, one should therefore construct a suitable ``copar comonad" whose Eilenberg-Moore category is proven to be monoidal, and lift $\CC$ to this category. 

In case $H$ is finite dimensional, we have observed that $\CC\cong \Hom_k((H^*)_{par},-)$. Since \((H^*)_{par}\) 
is an algebra, its full dual \(((H^*)_{par})^*\) is a topological coalgebra. Moreover \(\Hom_k((H^*)_{par}, V) \cong V \hotimes ((H^*)_{par})^*\) 
as vector spaces (see \cite[Theorem 4.4]{takeuchi}, and the isomorphism is a homeomorphism because \(V\) is discrete. This suggests that one would be able to describe partial comodules over \(H\) as (discrete) topological comodules over this topological coalgebra. 

All this will be part of future investigations.

\subsection*{Acknowledgements}
For this research, JV would like to thank the FWB (f\'ed\'eration Wallonie-Bruxelles) for support through the ARC project ``From algebra to combinatorics, and back'' and the ``Fonds Thelam'' for support through the project ``Partial symmetries of Non-commutative spaces''. 
JV and WH thank Departamento de Matem{\'a}tica, Universidade Federal de Santa Catarina, Brazil and in particular Eliezer Batista and Felipe Castro for warm hospitality during their respective visits in November 2019, when this work was initiated and April 2022, when this work was finalized.
JV thanks Paolo Saracco for interesting discussions about partial (co)representations and being a \TeX-wizard.

\end{document}